\documentclass{article}
\RequirePackage{snapshot}
\usepackage[T1]{fontenc}
\usepackage{lmodern}
\usepackage{babel}
\usepackage{subfigure}
\usepackage{amsmath}
\usepackage{amssymb}
\usepackage{amsthm}
\usepackage{mathtools}
\usepackage{centernot}
\usepackage{float}
\usepackage{floatpag, mwe}
\usepackage{placeins}
\usepackage[nice]{nicefrac}
\usepackage{algorithm}
\usepackage{algpseudocode}
\usepackage{authblk}
\usepackage{hyperref}
\hypersetup{
    colorlinks=true,
    linkcolor=blue,
    filecolor=magenta,      
    urlcolor=cyan,
    pdfpagemode=FullScreen,
    }
\usepackage[nameinlink]{cleveref}
\newcommand{\R}{\mathbb{R}}
\newcommand{\N}{\mathbb{N}}

\newtheorem{theorem}{Theorem}
\newtheorem{proposition}{Proposition}
\newtheorem{lemma}{Lemma}

\newtheorem{definition}{Definition}
\newtheorem{remark}{Remark}
\newtheorem{example}{Example}
\title{Duality Perspective on Nonlinear Eigenproblems}
\author[1]{Jonathan Laubmann}
\author[2]{Manuel Friedrich}
\author[1]{Daniel Tenbrinck}
\affil[1]{Department of Data Science, Friedrich-Alexander-Universität Erlangen-Nürnberg}
\affil[2]{Institut für Analysis, Johannes Kepler Universität Linz}
\begin{document}
\hyphenation{ei-gen-func-tion ei-gen-prob-lem for-mu-la-tion dis-cre-ti-za-tions dis-cre-ti-za-tion ap-prox-i-ma-ted sub-dif-fe-ren-tial sta-ting dif-fer-en-tia-ble ho-mo-ge-nei-ty ei-gen-vec-tors ei-gen-vec-tor opera-tor mo-no-to-nously}
\maketitle
\begin{abstract}
We investigate nonlinear eigenproblems for a broad class of proper, closed, convex functionals in reflexive Banach spaces. We develop a dual formulation of the nonlinear eigenproblem using the Fenchel conjugate and establish an equivalence to the primal problem. Further, we introduce a duality gap and a geometric characterization of eigenvectors that apply in general Banach spaces. We interpret the dual problem as the eigenproblem for the inverse operator of the primal problem. Concerning numerical methods for solving nonlinear eigenproblems, we analyze the inverse power method, framed as a dual power method, showing strong convergence in the case of absolutely \(p\)-homogeneous functionals. Our theoretical results are validated by extensive numerical experiments for the \(p\)-Laplacian. We further connect the flow-based proximal power method from the literature to the inverse power method and discuss two numerical approaches to approximate higher-order nonlinear eigenfunctions. 
\end{abstract}
%
\tableofcontents
\section{Introduction}\label{sect:Intro}
Linear eigenproblems and their respective solutions have been studied for a long time, due to their relevance in many disciplines ranging from applications in signal processing to mechanics and physics.
Solutions to eigenproblems have many interesting use cases in data science and machine learning, e.g., unsupervised clustering \cite{spectral_clustering}, recommendation systems \cite{recommender_systems}, dimension reduction \cite{LLE_dim_reduction}, or functional mappings for shapes \cite{functional_mappings}.
In recent years, the concept of linear eigenproblems has been generalized to nonlinear eigenproblems allowing for new applications, for instance in image segmentation, denoising, or graph clustering \cite{Hein_Buehler, gilboa2018nonlinear, bungert2019computing}.
As applications in mathematics, we mention that optimal Poincaré or embedding constants can be seen as nonlinear eigenvalues \cite{HYND2017approxrq}.
Nonlinear and linear eigenproblems can be formulated in finite and infinite dimensions as well as on discrete structures like graphs \cite{Daniel_Graph}.
An abstract theory that covers both linear and nonlinear settings can help to unify approaches and formulate efficient numerical approaches for a wide range of applications.

Nonlinear eigenproblems aim to find an element \(u\) of a vector space \(X\) and a scalar \(\lambda\in \R\) such that for two possibly nonlinear functionals \(T, S\) the equation
\begin{equation}\label{eq:abstract_eigenproblem}
    T(u) = \lambda S(u)
\end{equation}
is fulfilled, see \cite{gilboa2018nonlinear}.
For a concrete realization of \eqref{eq:abstract_eigenproblem} see \Cref{ex:p-laplace}.

The goal of this paper is to analyze the nonlinear eigenproblem for sub\-differentials of a convex functional $J$ on a reflexive Banach space \(X\), equipped with a norm \(\|\cdot\|_X\). We use the class \(\Gamma_0(X)\) of proper, lower semi-continuous, and convex functionals and fix the functional \(J\in \Gamma_0(X)\) as an element of this class.
The nonlinear operator \(T \coloneqq \partial J\)   in \eqref{eq:abstract_eigenproblem} is defined as the possibly set-valued subdifferential operator of \(J\), where the elements of \(\partial J\) are in the dual space \(X^*\), equipped with the respective dual norm \(\|\cdot\|_*\).
Therefore, we fix the role of the second operator \(S\) in \eqref{eq:abstract_eigenproblem} as a suitable duality mapping \cite{Chidume2009dualitymaps}.
In particular, we choose \(H\in\Gamma_0(X)\) to be absolutely \(p\)-homogeneous, i.e., the \(p\)-th power of a semi-norm on \(X\) with \(1<p<\infty\).
(Note that \(H^\frac{1}{p}\) induces a semi-norm, see \Cref{lemma:phomog_seminorm}.)
Following a common setup for nonlinear eigenproblems \cite{HYND2017approxrq, Bungert2022GFburger}, without restriction we replace \(X\) with the quotient space \(X\setminus \mathcal{N}(H)\), see \cite[Appendix A]{Bungert2020asympt}, which allows us to work with real norms.
In particular, this results in \(H\) having a trivial null space.
In this setting, the choice of \(H\) is naturally the $p$-th power of a standard isotropic norm on \(X\).
We define this norm as \(|\cdot|_H\coloneqq (pH)^\frac{1}{p}(\cdot)\), thus we write \(H(u) = \frac{1}{p}|u|^p_H\).
The duality mapping \(S\) is  naturally chosen as the duality mapping with gauge \(\phi(t) = t^{p-1}\) and the norm \(|\cdot|_H\), namely the subdifferential of \(H\) \cite{Chidume2009dualitymaps}.
With the above definitions, we can now formulate the nonlinear eigenproblem investigated in this work.
\begin{definition}[\(p\)-eigenvector]\label{def:eigenproblem}
We call \(u\in X\setminus\{0\}\) a \(p\)-eigenvector of \(J\) with subgradient \(\zeta\in  \partial J(u) \subset X^*\), duality map \(\partial H\), and eigenvalue \(\frac{J(u)}{H(u)} = \lambda\in\R\) if it fulfills the inclusion
\begin{equation}\label{eq:eigenproblem}
\zeta - \lambda\, \partial H(u) \ni0\,.
\end{equation}
\end{definition}
We will omit the dependency on \(H\) if it is clear from the context and simply state that \(u\) is an eigenvector of \(J\) if \(u\) satisfies \Cref{def:eigenproblem}.
\begin{example}\label{ex:p-laplace}
    A popular example that is covered by our formulation in \eqref{def:eigenproblem} is the \(p\)-Laplace eigenproblem for \(1<p<\infty\) with \(J\) being the \(p\)-Dirichlet energy, \(\frac{1}{p}\int_\Omega |\nabla u|^p\,dx\), the isotropic norm is chosen as the \(L^p\)-norm and \(H(u) = \frac{1}{p}\|u\|_{L^p(\Omega)}^p\) on \(X=W^{1,p}_0(\Omega)\). This leads to an eigenproblem in form of a partial differential equation (PDE), namely
    \begin{equation}\label{eq:p-laplace}
        -\Delta_p u - \lambda |u|^{p-2}u = 0\,.
    \end{equation}
    It is well known that for eigenvectors of absolutely \(p\)-homogeneous functionals, such as the \(p\)-Dirichlet energy, the equality \( J(u) = \lambda H(u)\) holds, see  e.g.\ \cite[Proposition 2.5]{BungertPHD}.
\end{example}
\begin{example}\label{ex:spd}
The simplest example covered by our formulation in \Cref{def:eigenproblem} is the linear eigenproblem for symmetric positive definite matrices $A$.
    The quadratic functionals \(J(u) = \frac{1}{2}\langle Au, u\rangle\) and \(H(u) = \frac{1}{2}\|u\|_2^2\) with the Euclidean norm \(\|\cdot\|_2\) and scalar product $\langle \cdot, \cdot \rangle$ in \(\R^n\) yield the linear eigenproblem
\begin{equation*}
    Au - \lambda u=0\,.
\end{equation*}
\end{example}

An important concept that is closely related to both linear and nonlinear eigenproblems is the so-called Rayleigh quotient (RQ).
\begin{definition}[Rayleigh quotient]\label{def:RQ}
    We call the mapping \(R:X\setminus\{0\}\to\R\cup\{\pm\infty\}\) defined by
    \begin{equation}\label{eq:RQ}\tag{RQ}
        R(u) = \frac{J(u)}{H(u)}\,
    \end{equation}
    the Rayleigh quotient of \(J\).
\end{definition}
The nonlinear RQ is a direct generalization of the linear RQ \[R(u) = \frac{\langle Au, u\rangle}{\langle u,u\rangle\,},\] 
where $A$ is chosen as in \Cref{ex:spd}. We can interpret the nonlinear eigenproblem in \Cref{def:eigenproblem} as the critical point problem for the nonlinear RQ \cite{Hein_Buehler}.
To illustrate that eigenvectors can be seen as critical vectors of the RQ, we assume for the sake of simplicity that the RQ is differentiable. Then, by using the quotient rule, for $v \in X$ we have the first variation
\begin{equation}\label{eq:first_variation_rq}
    \frac{d}{dt}R(u+vt)|_{t=0} = \frac{1}{H(u)}\left(\langle \partial J(u) - R(u)\partial H(u),v\rangle \right)\,,
\end{equation}
where \(\langle\cdot,\cdot\rangle:X^*\times X\to\R\) denotes the dual pairing of elements in \(X^*\) and \(X\).
We see that the term in \eqref{eq:first_variation_rq} is zero for all \(v\in X\) if and only if \(u\) fulfills the \(p\)-eigenvector condition \eqref{eq:eigenproblem} with \(R(u)=\lambda\).

Throughout this paper we pose a growth assumption on the functionals \(J\) and \(H\), which can be seen as a coercivity assumption on \(J\) with respect to the norm \(|\cdot|_H\). We assume that there exists a constant \(c>0\) such that
\begin{equation}\label{eq:growth_coercivity}
    H(u) \leq c\,J(u)\quad \forall u\in X\,.
\end{equation}
The growth assumption ensures the nonnegativity of \(J\) and boundedness of the RQ away from zero. In the following, we will mainly focus on the special case of \(J\) being absolutely \(p\)-homogeneous.
As explained  in \Cref{lemma:phomog_seminorm}, this class of functionals is given as the power of semi-norms, which appear, e.g., as regularization terms in denoising tasks \cite{Benninger_Burger_GS_SV}. 
In this case,  without further notice we set \(\mathcal{N}(J) = \{0\}\) by replacing $X$ by $X\setminus \mathcal{N}(J)$ as otherwise \eqref{eq:growth_coercivity} would be immediately violated.
Besides boundedness of the RQ from below, a further  consequence of \eqref{eq:growth_coercivity} is that eigenvalues of \(J\) are always positive, because they are determined by the RQ value of the respective eigenvector.
The global minimizer of the RQ, if it exists, is called the \emph{ground state} \(u_* \coloneqq {{\rm arg\, min}_{u\in X\setminus\{0\}}} \frac{J(u)}{H(u)}\).
The ground state is an eigenvector with eigenvalue \(\lambda_* = R(u_*)\), even in the non-differentiable setting \cite{Bungert2022GFburger}.
The ground state's respective eigenvalue is the reciprocal optimal growth constant in \eqref{eq:growth_coercivity}.
\subsection*{Related work and contributions of the paper}
Nonlinear eigenproblems have gained research interest in recent years.
In the following we give a brief overview of related works in this field. 
Theoretical analysis of nonlinear eigenproblems has been performed in various recent works, see, e.g., \cite{BungertPHD, Bungert2022GFburger, Bungert_Korolev_infty, bungert2025introductionnonlinearspectralanalysis, deidda2024graphinftylaplacianeigenvalueproblem, deidda2025nonlinearspectralgraphtheory, Ercole2018abstract_NonLinEigValPb, gilboa2018nonlinear, HYND2017approxrq, Juutinen1999}.
Asymptotic profiles of gradient flows were shown to satisfy the nonlinear eigenvector equation \eqref{eq:eigenproblem}, and thus were analyzed in the context of nonlinear eigenproblems in \cite{Bungert2020asympt, Bungert2021spectrdec, HYND2017approxrq, Varvaruca2004evoequ}.
Dual formulations of the nonlinear eigenproblem can be found in \cite{Bungert_Korolev_infty, tudisco2022nonlinearspectralduality}.
Notes on the \(p\)-Laplacian, a mean value approximation, and its formulation on graphs are given in \cite{Lindqvist_notesonplpleq, Teso2021meanvalue, delTeso2021, Daniel_Graph}.

On the other hand, numerous works investigated numerical schemes to ap\-prox\-i\-mate nonlinear eigenfunctions in finite dimensions.
Approaches for the computation of nonlinear eigenfunctions for specific problems can be found in \cite{Bozorgnia_2024, Bungert2021NNpowermethod, ge2025computingplaplacianeigenpairssigned, Lanza2024peigendec, Li_2025}.
One can divide these approaches into two different paradigms, namely \emph{flow-based methods} \cite{Aujol2018_onehomogflow, Bungert2022GFburger, bungert2019computing, COHEN2018Gilboa_flow, Feld_2019, Gilboa2021iterativemeth, NossekG16flowgenerating} and \emph{variants of (inverse) power methods} \cite{bobkov2025inverseiterationmethodhigher, Farid_B_plpl, Farid_B_2ndplpl, Hein_Buehler, shao2024simpleinversepowermethod}.

In the following, we state the main contribution of our work and discuss how the proposed approach compares to recent work in the field.
In the first part of this paper, we formulate a nonlinear eigenproblem for the broad class of proper, closed, convex functionals in reflexive Banach spaces. The majority of related works has additional assumptions on the investigated functional or restricts the problem to Hilbert spaces or even finite dimensions. To the best of our knowledge, so far the problem has been analyzed in a general setting only in a seminal paper by Bungert and Burger  \cite{Bungert2022GFburger}. Yet, our formulation differs from the one by Bungert and Burger \cite{Bungert2022GFburger}. More precisely,  \cite{Bungert2022GFburger}  regards eigenvectors as tuples \((u, \hat{u})\in X\times \operatorname{arg\,min}J\) and respective eigenvalues \(\lambda\in \R\), and addresses the problem
\[\partial J(u) - \lambda\partial H(u - \hat{u})\ni 0\ \quad \text{with}\quad R(u) = \frac{J(u)-J(\hat{u})}{H(u-\hat{u})},\]
with $u \neq \hat{u}$,
where the shifting by $\hat{u}$ avoids minimizers of \(J\) to be trivial ground states, by maximizing over all \(\bar{u}\in \operatorname{arg\, min}J\) in the ground state problem. In our setting, instead, we work under the coercivity assumption \eqref{eq:growth_coercivity} which yields a positive RQ and, in turn, allows only positive eigenvalues since they are given as the RQ value of the eigenvector. Note that the two formulations coincide in the case that \(J\) is absolutely \(p\)-homogeneous, because \(\operatorname{arg\,min}J = \mathcal{N}(J)\).

We introduce a dual formulation of the nonlinear eigenproblem based on the well-known Fenchel conjugate and show equivalence of the primal and dual problem. More precisely, we confirm that \(u \in X\) is a \(p\)-eigenvector of \(J\) with eigenvalue \(\lambda\) and subgradient \(\zeta \in X^*\) if and only if \(\zeta\) is a dual eigenvector of \(J^*\) with eigenvalue \(\lambda^{1-q}\) and subgradient \(u\in X\), where \(q = \frac{p}{p-1}\) is the corresponding  Hölder exponent.
In  \cite{tudisco2022nonlinearspectralduality},  Tudisco and Zhang  analyze the effect of duality transformations on the spectrum of positively \(p\)-homogeneous functionals in finite dimensions. They also establish a relation between Fenchel duality for a nonnegative \(p\)-homogeneous function \(J\) and norm duality in the sense of $\Gamma$-convergence in the limit \(p\to 1\). We are extending their work on the Fenchel transformation to the class of proper, closed, convex functionals on reflexive Banach spaces.
Bungert and Korolev \cite{Bungert_Korolev_infty} analyze the eigenproblem for \(p=1\) and \(J\) as a norm on a Banach space \(X\). In particular, they formulate a dual problem using norm duality.
 
We introduce a formulation which we denote as dual Rayleigh quotient and show that it is bounded from above if and only if the primal RQ is bounded from below by a positive constant.
A comparable result can be found in the work of Bungert and Korolev \cite{Bungert_Korolev_infty}, for the special case of \(J\) being a norm.
In particular, we demonstrate that the ground state problem, i.e., minimizing the RQ in \eqref{eq:RQ}, is equivalent to maximizing a dual RQ.
Furthermore, we introduce a so-called duality gap as a novel metric for the computation of nonlinear eigenvectors.
We can show that the roots of the duality gap are in one-to-one correspondence with a primal-dual eigenvector pair.
We also interpret the dual problem as the eigenproblem of the inverse operator.
 
Furthermore, we introduce a geometric characterization of eigenvectors $u$, stating that they are given as vectors on a sphere with radius chosen as their respective norm \(|u|_H\), for which the subgradient $\zeta$ of \(J\) is normal to this sphere at $u$ and has length $ |\zeta|_{H^*} = \frac{pJ(u)}{|u|_{H}}$.
This characteristic applies to a broader class of functionals and also holds in a more general Banach space than the previous characteristic given by Bungert et al.\ in \cite{BungertPHD, Bungert2021spectrdec}.
Their characterization is given in Hilbert spaces for eigenvectors of semi-norms and states that eigenvectors \(\zeta\in\partial J(\zeta)\)
in the dual unit ball of \(J\) are precisely
those vectors in the dual unit ball which possess an orthogonal hyperplane supporting this dual unit ball in \(\zeta\).
Our characterization of nonlinear eigenfunction is based on the cosine similarity measure for the norm \(|\cdot|_H\) and will be used as a numerical metric to show convergence for our iterative schemes and as an objective function in numerical optimization.

In the second part of this work, where we discuss numerical methods to compute nonlinear eigenfunctions, we restrict ourselves to absolutely \(p\)-homo\-ge\-neous functionals.
We formulate the inverse power method as a dual power method and show strong convergence for a subsequence of the iteration sequence towards nonlinear eigenfunctions.
The inverse power method for the general class of convex and absolutely \(p\)-homogeneous functionals in Banach spaces was in\-vest\-i\-gated by Ercole  \cite{Ercole2018abstract_NonLinEigValPb}.
Our analysis uses the dual formulation of the problem as we see the inverse power method as a dual power method.
We show that the proximal power method by Bungert and Burger \cite{Bungert2022GFburger}, a flow based method, falls under the broad class of inverse power methods in our reflexive Banach space setting.

Finally, we discuss two numerical approaches to approximate higher-order nonlinear eigenfunctions.
The first approach was proposed by Bobkov and Galimov for the computation of higher-order  eigenfunctions of the \(p\)-Laplacian \cite{bobkov2025inverseiterationmethodhigher}.
In our numerical experiments, we show that the plateauing effect of their scheme, which the authors reported in their work, is due to the discretization and inaccuracy of the underlying finite element scheme and can be avoided by using the del Teso and Lindgren finite difference solver \cite{delTeso2021}.
Secondly, we propose a new method for the computation of higher-order eigenfunctions which is not restricted to the \(p\)-Laplacian and can be used for other eigenproblems as well.
For this, we are using a target function based on our geometric characterization instead of the RQ.  With this approach, we are not dependent on characteristics of specific eigenfunctions, like being non-sign changing or that the RQ of the positive and negative part of the eigenfunction is equal, but solely on the initial datum. 
Our approach  keeps characteristics of the initial datum and thus enables us to converge towards a `closer' eigenfunction in terms of the cosine similarity and thus to higher eigenfunctions as well.
\subsection*{Structure of this work}
Our paper is structured as follows.
We finish \Cref{sect:Intro} by introducing the necessary mathematical notation for our analysis.
In \Cref{sect:DualProblem} we introduce an abstract duality theory for nonlinear eigenproblems and state our main results, namely the equivalence of the primal and dual problem, the geometric characterization of \(p\)-eigenvectors, and the proof that eigenvectors are roots of the duality gap.
In \Cref{sect:GFandIPM} we focus on  numerical approximation schemes for nonlinear eigenfunctions.
We connect the gradient flow based approach by Bungert and Burger to the inverse power method and prove convergence of the inverse power method.
In \Cref{sect:Numeric} we empirically verify the convergence of the inverse power method and our theoretical results.
Additionally, we discuss approaches to compute higher-order eigenfunctions and explain the plateauing behavior of the Bobkov and Galimov scheme \cite{bobkov2025inverseiterationmethodhigher}.
We end the section on the computation of nonlinear eigenfunctions by introducing and empirically testing an approach based on the geometric characterization of nonlinear eigenvectors, which aims to compute higher-order eigenfunctions.
We conclude this work with a discussion of our main results and open research questions in \Cref{sect:Conclusion}.
\subsection*{Notation and prerequisites from convex a\-nal\-y\-sis} \label{subsection:math_notation}

As before, by $X$ we denote a reflexive Banach space, and  by $\Gamma_0(X)$ we denote the class of  proper, lower semi-continuous,
and convex functionals. Our notation mainly follows the terminology in   standard textbooks on convex analysis \cite{Rockafellar1970convana, Ekeland99convana}. The effective domain of \(J\) and the null-space of \(J\) are defined as
\begin{align*}
    \operatorname{dom}(J) &\coloneqq \{u\in X\mid J(u)<\infty\}\,,\\
    \mathcal{N}(J) &\coloneqq \{u\in X\mid J(u)=0\}\,.
\end{align*}
The functional \(J\in\Gamma_0(X)\) is continuous on its effective domain \cite{Ekeland99convana}.
\begin{definition}\label{def:Fenchel_conj}
    The Fenchel conjugate \(J^*\colon X^*\to\R\) of   \(J \in \Gamma_0(X)\) is given by \[J^*(\zeta) = \sup_{u\in X}\langle \zeta, u\rangle - J(u)\,.\]
\end{definition}
The Fenchel conjugate is sometimes called convex conjugate or polar function. We have that \(J^*\in \Gamma_0(X^*)\) for \(J\in \Gamma_0(X)\).
\begin{definition}
    The subdifferential of a convex functional \(J\) is given by the set-valued mapping \(\partial J:X\rightrightarrows X^*\)
    \[\partial J(u) = \left\{\zeta\in X^*\mid \langle\zeta, v-u\rangle + J(u)\leq J(v)\quad \forall v\in X\right\}.\]
\end{definition}
Note that its elements, called subgradients, are elements in the dual space. A useful relationship is the Fenchel-Young identity
\begin{equation}\label{eq:Fenchel_Young_eq}
    \zeta\in \partial J(u)\: \Leftrightarrow \: u\in\partial J^*(\zeta) \: \Leftrightarrow \: \langle\zeta,u\rangle \, = \, J(u) + J^*(\zeta).
\end{equation}
Furthermore, for any \(u\in X\) and \(\zeta\in X^*\), the Fenchel-Young inequality  \(\langle\zeta,u\rangle \leq J(u) + J^*(\zeta)\) holds.

In this work, we focus mainly on absolutely \(p\)-homogeneous functionals in \(\Gamma_0(X)\).
We now summarize the most important properties of this class of functionals.
If \(H\) is absolutely \(p\)-homogeneous with \(p\geq 1\) and convex, then \(\operatorname{dom}(H)\) and \(\mathcal{N}(H)\) are linear subspaces.
If \(H\) is additionally lower semi-continuous, then \(\mathcal{N}(J)\) is closed \cite{Bungert2020asympt}. It is well known that absolutely one-homogeneous functionals in \(\Gamma_0(X)\) are semi-norms \cite{Bungert2021spectrdec, BungertPHD}. We generalize this fact by showing that an absolutely \(p\)-homogeneous functional in \(\Gamma_0(X)\) is the \(p\)-th power of a semi-norm.
\begin{lemma}\label{lemma:phomog_seminorm}
    Let \(J \in \Gamma_0(X)\) be absolutely \(p\)-homogeneous. Then, \(J\) is the \(p\)-th power of a semi-norm.
\end{lemma}
\begin{proof}
    When taking the \(p\)-th root of an absolutely \(p\)-homogeneous functional in \(\Gamma_0(X)\), we have an absolutely one-homogeneous, proper and lower semi-continuous functional. It remains to prove that the convexity is preserved by taking the \(p\)-th root, see \cite{tudisco2022nonlinearspectralduality}. We need to show that: \(J\) is convex \(\Rightarrow\) \(J^\frac{1}{p}\) is convex.
    We first discuss the most relevant case \(u, v\in \operatorname{dom}(J^\frac{1}{p})\setminus \mathcal{N}(J^\frac{1}{p})\), i.e., \(0<J(u), J(v) <\infty\).
   With \(\infty>C=tJ^\frac{1}{p}(u) + (1-t)J^\frac{1}{p}(v)>0\) for $t \in [0,1]$,  we calculate
    \begin{align*}
        \frac{J(tu + (1-t)v)}{C^p} &= J\left(\frac{tJ^\frac{1}{p}(u)u}{CJ^\frac{1}{p}(u)} + \frac{(1-t)J^\frac{1}{p}(v)v}{CJ^\frac{1}{p}(v)}\right)\\
        &\leq \frac{tJ^\frac{1}{p}(u)}{C}J\left(\frac{u}{J^\frac{1}{p}(u)}\right) +\frac{(1-t)J^\frac{1}{p}(v)}{C}J\left(\frac{v}{J^\frac{1}{p}(v)}\right)\\
        &= \frac{tJ^\frac{1}{p}(u)}{C} + \frac{(1-t)J^\frac{1}{p}(v)}{C} = 1\,.
    \end{align*}
   By taking the \(p\)-th root this proves that \(J^\frac{1}{p}\) is convex on \(\operatorname{dom}(J^\frac{1}{p})\setminus \mathcal{N}(J^\frac{1}{p})\).
   As $\mathcal{N}(J)$ is a linear subspace, the case \(u, v\in \mathcal{N}(J^\frac{1}{p})\) is easy, and  the case \(u\in X\setminus\operatorname{dom}(J^\frac{1}{p})\) is trivial. Without loss of generality, it remains to treat the case \(u\in \mathcal{N}(J^\frac{1}{p})\) and \(v\in \operatorname{dom}(J^\frac{1}{p})\).
    We use that \(\mathcal{N}(J)\) is a closed linear subspace and that \(J(w+w_0) = J(w)\) for all \(w_0\in \mathcal{N}(J)\) and all \(w\in X\), see \cite[Appendix]{Bungert2020asympt}.
    We calculate
    \[J(tu+(1-t)v) = J((1-t)v) = (1-t)^pJ(v)\,.\]
    Taking the \(p\)-th root and using \(J(u) = 0\), we receive
    \[J^\frac{1}{p}(tu+(1-t)v) = (1-t)J^\frac{1}{p}(v)  = (1-t)J^\frac{1}{p}(v)+ tJ^\frac{1}{p}(u)\,,\]
    which concludes the proof. 
\end{proof}
For  absolutely \(p\)-homogeneous functionals \(H\in\Gamma_0(X)\) we define the induced semi-norm from \Cref{lemma:phomog_seminorm} as \(|\cdot|_H := (pH)^{\frac{1}{p}}(\cdot)\). 
In this case, its Fenchel conjugate is the normalized power of its dual norm \[|\zeta|_{H^*} \coloneqq \sup_{u\in X\setminus\{0\}}\frac{\langle\zeta, u\rangle}{|u|_H}\] with Hölder exponent \(q = \frac{p}{p-1}\), written as \(H^*(\zeta) = \frac{1}{q}|\zeta|_{H^*}^q\).
For \(\zeta\in\partial H(u)\) it holds that 
\begin{align}\label{newequation}
|\zeta|_{H^*} = |u|^{p-1}_H.
\end{align}
 The subdifferential of \(H\) can be seen as a duality map \(\partial H\colon X\rightrightarrows X^*\) \cite{Chidume2009dualitymaps}.
For any absolutely \(p\)-homogeneous functional \(H\in\Gamma_0(X)\) the Euler identity
\begin{equation}\label{eq:euler_ident}
    pH(u) = \langle \zeta, u\rangle\quad \text{for} \quad \zeta\in \partial H(u)
\end{equation}
holds, see \cite{YANG2008euler}.
Note that the Euler identity does not imply that the subdifferential is single-valued, because for \(\zeta_1, \zeta_2\in \partial H(u)\) it holds \[0= pH(u) - pH(u) = \langle \zeta_1 - \zeta_2, u\rangle \centernot\Rightarrow\zeta_1=\zeta_2\,.\]
In fact, \(H\) is differentiable on its effective domain if and only if \(\partial H\) is single-valued \cite{Ekeland99convana}.
Additionally, we   see, by using its induced norm, that \(H\) is differentiable if \(|\cdot|_H\) is differentiable as composition of differentiable functionals.
However, the converse is not true, e.g., the absolute value in \(\R\) and its square.
The norm of a uniformly smooth Banach space is smooth, except at zero, e.g., any \(W^{k,p}\) norm with \(1<p<\infty\).
\begin{remark}
    The assumption that the space \(X\) is reflexive is not inherently necessary for the ideas presented in this work.
    Yet, the reflexivity of \(X\) is useful for the analysis in the sense that we do not need to specify whether we are working in the predual or the dual space.
    Furthermore, the dual transformation to the dual space and the transformation from the dual space to the primal space are equivalent on reflexive Banach spaces.
\end{remark}
\section{Duality of the nonlinear eigenproblem}\label{sect:DualProblem}
In this section, we present a duality theory for nonlinear eigenproblems based on the Fenchel conjugate.
Furthermore, we show the equivalence of the primal and dual problem in the sense that \(u \in X\) is a \(p\)-eigenvector of \(J\) with eigenvalue \(\lambda\) and subgradient \(\zeta \in X^*\) if and only if \(\zeta\) is a dual eigenvector of \(J^*\) with eigenvalue \(\lambda^{1-q}\) and subgradient \(u\in X\).
The duality theory introduced here gives rise to a geometric characterization of nonlinear eigenvectors, and we will interpret the dual problem as the eigenproblem of the respective inverse operator of the primal operator \(\partial J\).
For the sake of clarity, we note again that in our setting we assume \((X, \|\cdot\|_X)\) to be a reflexive Banach space, \(J\in \Gamma_0(X)\), and \(H\in \Gamma_0(X)\) to be absolutely \(p\)-homogeneous for \(p>1\).

\subsection{Dual problem}\label{subsect:dual_problem}
Duality in optimization often refers to the relationship between a minimization problem and its corresponding dual problem, which is usually a maximization problem.
We will show that this is the case for the ground-state eigenvector which minimizes the primal RQ, whereas its dual eigenvector maximizes the dual RQ.\\
We define the dual eigenproblem on the dual space using the Fenchel conjugate of the functionals \(H\) and \(J\) from \Cref{def:eigenproblem}.
In \cite{Bungert_Korolev_infty} the authors introduced a dual formulation for the \(1\)-homogeneous eigenproblem in Banach spaces based on dual norms.
In \cite{tudisco2022nonlinearspectralduality} the spectrum of positively \(p\)-homogeneous convex functions in finite dimensions was analyzed with respect to different duality transforms.
We are extending their analysis with respect to the Fenchel conjugate to a broader class of convex functionals in infinite-dimensional Banach spaces.
We recall \Cref{def:RQ} and formulate the corresponding definition for the dual RQ.
\begin{definition}[Dual Rayleigh quotient]\label{def:dRQ}
    We define the dual Rayleigh quotient as the mapping \(R_*:X^*\setminus \{0\}\to \R\cup\{\pm\infty\}\) given by
    \begin{equation*}\tag{dRQ}
        R_*(\zeta) = \frac{J^*(\zeta)}{H^*(\zeta)}\,.
    \end{equation*}
\end{definition}
We note that the dual RQ is not the Fenchel conjugate of the primal RQ, but the quotient of the Fenchel conjugates \(J^*\) and \(H^*\).
The primal RQ is positive, due to the boundedness from below, see \eqref{eq:growth_coercivity}.
However, the dual RQ can be negative, because \(J^*\) can be negative in principle. For example, if we choose \(X=\R_{>0}\) and \(J(u) = e^{u}\), then its Fenchel conjugate is given by \(J^*(\zeta) = \zeta\log(\zeta) - \zeta\), which is -1 for \(\zeta=1\).
We define the dual eigenproblem analogously to the primal problem \Cref{def:eigenproblem}, but with the Fenchel conjugates \(J^*\) and \(H^*\).
\begin{definition}[\(q\)-eigenvector]
We call \(\zeta\in X^*\setminus\{0\}\) a \(q\)-eigenvector of \(J^*\) with subgradient \(u\in \partial J^*(\zeta) \subset  X\), duality map \(\partial H^*\), and eigenvalue \(R_*(\zeta) = \mu\in\R\) if it fulfills the inclusion 
\begin{equation*}
    u - \mu\, \partial H^*(\zeta) \ni 0\,.
\end{equation*}
\end{definition}
\begin{remark}
    Within our abstract framework, the existence of nonlinear eigen- vectors is not guaranteed.  Taking for example the one dimensional case \(X=\R\) with \(J(u) = e^{u^2}-1\) and \(H(u) = \frac{1}{2}u^2\) we see that the growth assumption \eqref{eq:growth_coercivity} is fulfilled,
    but the only nonlinear eigenvector would be the trivial zero vector which is excluded.
    When discussing results for ground states or general eigenvectors, we implicitly assume their existence hereinafter. For the existence theory of infinitely many distinct nonlinear eigenfunctions in infinite dimensional Banach spaces we refer to the Lusternik-Schnirelmann theory \cite{Lusternik_Schnirelmann_1934}, see \cite{Amann1972_LusternikSchnirelmann} and the re\-fe\-ren\-ces therein.
\end{remark}
The two eigenproblems are strongly related.
Indeed, in the mathematical setting we introduced above the primal and dual eigenproblem are equivalent.
We can use the Fenchel-Young equality \eqref{eq:Fenchel_Young_eq} for \(J\) and \(J^*\) as well as the homogeneity of \(H\) and \(H^*\) with dual Hölder exponents \(p\) and \(q\) to connect primal and dual eigenvectors and values.
\begin{theorem}[Eigenvector duality]\label{thm:duality}
    The primal and dual eigenvector problem are equivalent in the sense that \(u\in X\) is a \(p\)-eigenvector of $J$ with subgradient \(\zeta\) and eigenvalue \(\lambda\), if and only if \(\zeta\) is a \(q\)-eigenvector of $J^*$ with subgradient \(u\) and dual eigenvalue \(\mu = \lambda^{1-q}\).
\end{theorem}
\begin{proof}
    Let \(u\in X\).
    Then, by reflexivity of \(X\) it holds
    \begin{align}\label{eq:proof_evec_duality}
        \zeta \in \partial J(u) 
        \Leftrightarrow u \in \partial J^*(\zeta) 
    \end{align}
    and 
    \[
        \eta \in \partial H(u) 
        \Leftrightarrow u \in \partial H^*(\eta).
    \]
    We fix \(u\) and \(\zeta\) such that \eqref{eq:proof_evec_duality} holds. Further, we let \(\lambda>0\) and set \(\eta=\frac{\zeta}{\lambda}\).
    The eigenvector condition \(\zeta\in\lambda\partial H(u)\) can equivalently be written as
    \[u \in \partial H^*(\eta) = \partial H^*\left(\frac{\zeta}{\lambda}\right) = \lambda^{1-q}\partial H^*(\zeta)\,,\]
    where we used that  $H^*$  is absolutely  \(q\)-homogeneous, see before \eqref{eq:euler_ident}. 
    Setting \(\mu=\lambda^{1-q}\), this is again equivalent to the dual eigenvector condition \( \partial J^*(\zeta)\ni u \in\lambda^{1-q}\,\partial H^*(\zeta) = \mu\, \partial H^*(\zeta)\).
    
    It remains to show that \(\lambda = R(u)\Leftrightarrow \mu = R_*(\zeta)\).
    Using the Euler identity \eqref{eq:euler_ident} and the Fenchel-Young identity \eqref{eq:Fenchel_Young_eq} along with \(\eta\in\partial H(u)\) we get
    \begin{align*}
    \frac{J(u)}{H(u)} &= \frac{pJ(u)}{\langle\eta, u\rangle} = \frac{\lambda pJ(u)}{\langle\zeta, u\rangle} = \lambda p\frac{\langle\zeta, u\rangle - J^*(\zeta)}{\langle\zeta, u\rangle} = \lambda p\left[1-\frac{J^*(\zeta)}{\langle\zeta, u\rangle}\right] \\
    &=\lambda p\left[1-\frac{J^*(\zeta)}{\lambda\langle \eta, u\rangle}\right]= \lambda p\left[1-\frac{J^*(\zeta)}{\lambda qH^*(\eta)}\right]= \lambda p\left[1-\frac{J^*(\zeta)}{\lambda^{1-q}qH^*(\zeta)}\right]\,.
    \end{align*}
    We finally have
    \[\frac{1}{p\lambda}\frac{J(u)}{H(u)} + \frac{1}{q\lambda^{1-q}}\frac{J^*(\zeta)}{H^*(\zeta)} \: = \: 1\,,\]
    which is the desired property.
\end{proof}
We have seen above that the dual RQ can become negative, but the dual eigenvalues are always positive, due to the fact that the primal eigenvalues are always positive. 
The eigenvalues of the primal and dual problem are mutually reciprocal with exponent \(1-q<0\), i.e., the smallest primal eigenvalue is related to the largest dual eigenvalue.
The smallest eigenvalue is the eigenvalue corresponding to the ground-state eigenvector and is connected to the optimal constant in the growth assumption \eqref{eq:growth_coercivity}. 
The latter reverses for the dual functional in comparison to the primal.
\begin{lemma}[Boundedness of the dual Rayleigh quotient]\label{lemma:boundedness}
    The primal Rayleigh quotient \Cref{def:RQ} is sharply bounded from below by the ground-state eigenvalue \(\lambda_*\) if and only if the dual Rayleigh quotient is sharply bounded from above by the dual eigenvalue \(\mu_* \coloneqq\lambda_*^{1-q}\).
\end{lemma}
\begin{proof}
Let \(\lambda_*\) be the global minimum of the primal RQ.
For any \(\zeta\in X^*\), with the definition of the Fenchel conjugate \Cref{def:Fenchel_conj} and the growth assumption \eqref{eq:growth_coercivity} we get
    \begin{align*}
        J^*(\zeta) &= \sup_{u\in X}\,\langle \zeta, u\rangle - J(u) \leq \sup_{u\in X}\,\langle \zeta, u\rangle - \lambda_*H(u) \\ &= (\lambda_*H)^*(\zeta) = \lambda_*H^*\left(\frac{\zeta}{\lambda_*}\right)=\lambda_*^{1-q}H^*(\zeta).
    \end{align*}
The boundedness assumption on the primal RQ translates to a bound from above for the dual RQ due to the order reversing of the Fenchel conjugates.
Based on the reflexivity property the reverse statement can be shown equivalently.\\
The inequality becomes sharp for \(\zeta_*\in \partial J(u_*)\), where $u_*$ denotes the ground-state eigenvector of the primal eigenproblem.
We use \Cref{thm:duality} to argue that the maximum is attained for the dual eigenvector \(\zeta_*\) with subgradient   \(u_*\). 
\end{proof}
\Cref{lemma:boundedness} shows that the maximum of the dual RQ is attained by the eigenvalue \( \mu_*\coloneqq \lambda_*^{1-q}\) of the dual of the ground-state eigenvector.
With the same argument, we see that boundedness of the dual RQ from below by a positive constant holds if and only if the primal RQ is bounded from above.
Boundedness from below by zero of the dual RQ can be obtained with the additional assumption that \(J\) is absolutely \(p\)-homogeneous, since then  \(J^*\) is a  power of a norm, and hence  nonnegative. 
\begin{example}[Optimal norm inequality]\label{ex:optimal_norm_ineq}
We can write an absolutely \(p\)-homo\-ge\-neous functional \(J\) as the power of a norm \(J(u) = \frac{1}{p}|u|_J^p\), by \Cref{lemma:phomog_seminorm},  with dual \(J^*(\zeta) = \frac{1}{q}|\zeta|_{J^*}^q\), where \(|\cdot|_{J^*}\) is the dual norm of \(|\cdot|_J\).
The growth assumption  \eqref{eq:growth_coercivity} can be seen as a norm inequality
\[|u|_H\leq \lambda_*^{-\nicefrac{1}{p}} |u|_J\,.\]
The duality theory \Cref{lemma:boundedness} now states that there exists a dual norm inequality
\[|\zeta|_{J^*}\leq\lambda_*^{\nicefrac{(1-q)}{q}}|\zeta|_{H^*}\,.\]
We see that the optimal constant is an eigenvalue for the ground-state eigenvector.
The primal and dual optimal constants are equal since \(\frac{1-q}{q} = -\frac{1}{p}\).\\
For a bounded domain \(\Omega\) and \(X=W^{1,p}_0(\Omega)\) with \(J(u) = \int_\Omega|\nabla u|^p\,dx\) and \(H(u) = \|u\|_{L^p(\Omega)}^p\), we obtain the Poincaré inequality, where \(\lambda_*^{-\nicefrac{1}{p}}\) is the optimal Poincaré constant depending only on the domain \(\Omega\) and \(p\), namely (see  \cite{HYND2017approxrq})
\[\|u\|_{L^p(\Omega)} \leq \lambda_*^{-\nicefrac{1}{p}}\|\nabla u\|_{L^p(\Omega)}\,.\]
Its dual counterpart is given by the inequality 
\[\|\zeta\|_{W^{-1,q}(\Omega)}\leq \lambda_*^{-\nicefrac{1}{p}} \|\zeta\|_{L^q(\Omega)},\]
where the norm \(\|\cdot\|_{W^{-1,q}(\Omega)}\) is given by 
$$\|\zeta\|_{W^{-1,q}(\Omega)} = \sup\left\{\langle \zeta,u\rangle\mid u\in W^{1,p}_0,\,\|\nabla u\|_{L^p(\Omega)}=1  \right\}.$$
\end{example}
\begin{example}[Maximizers of the Rayleigh quotient are eigenvectors]\label{ex:max_rq}
    We show that the maximizer \(\bar{u}\) of   a RQ, under the assumption that  this  RQ is  from above and that such maximizer exists, is an eigenvector. 
   This could for example be applied to the dual RQ (which is bounded from above) if there exists a maximizer.
We set
\[\bar{\lambda}\coloneqq\max_{u\in X\setminus\{0\}} \frac{J(u)}{H(u)}\,,\]
i.e., \(J(\bar{u}) = \bar{\lambda}H(\bar{u})\), and we fix \(\bar{\zeta}\in \partial J(\bar{u})\).
We prove that \(\bar{u}\) is a \(p\)-eigenvector with subgradient 
\(\bar{\zeta}\) and eigenvalue \(\bar{\lambda}\).
By the definition of the subgradient we have for any \(u\in X\)
\begin{align*}
    \langle\bar{\zeta}, u-\bar{u}\rangle + J(\bar{u}) \leq &J(u) \\
    \Leftrightarrow  \langle\bar{\zeta}, u-\bar{u}\rangle \leq J(u) - &J(\bar{u}) = J(u) - \bar{\lambda}H(\bar{u})\leq \bar{\lambda}(H(u) - H(\bar{u}))\,.
\end{align*}
Dividing by \(\bar{\lambda}\) gives \(\frac{\bar{\zeta}}{\bar{\lambda}}\in \partial H(\bar{u})\), which proves that \(\bar{u}\) is a \(p\)-eigenvector with subgradient \(\bar{\zeta}\) and eigenvalue \(\bar{\lambda}\).
\end{example}
\begin{remark}[Minimal norm element]\label{lem:minimal_norm}
    Let \(J\) be absolutely \(p\)-homogeneous.  Then, the dual eigenvector \(\zeta\in X^*\) of a primal eigenvector \(u\in X\) with eigenvalue $\lambda$ is an element of minimal \(|\cdot|_{H^*}\)-norm in the subgradient of \(\partial J(u)\).\\
To prove this statement, similarly as in \cite{Bungert2021spectrdec}, for any \(\tilde{\zeta}\in\partial J(u)\), by using \eqref{newequation}, \eqref{eq:euler_ident},  and the fact that  \(\frac{\zeta}{\lambda}\in\partial H(u)\), we calculate 
    \begin{align*}
        |u|_H|\zeta|_{H^*} = \lambda|u|_H|\frac{\zeta}{\lambda}|_{H^*} = \lambda |u|_H|u|_H^{p-1} =\lambda p H(u) = \langle\zeta, u\rangle = pJ(u) =\langle \tilde{\zeta} , u\rangle \,.
    \end{align*}
    By Cauchy–Schwarz this gives us $ |u|_H|\zeta|_{H^*} \le   |u|_H|\tilde{\zeta}|_{H^*}$.
\end{remark}
Following \Cref{thm:duality}, if \(u\) and \(\zeta\) are dual eigenvectors, then their respective eigenvalues $\lambda$ and $\mu$ (thus their Rayleigh quotients) are related by \(\lambda^{-\nicefrac{1}{p}} = \mu^{\nicefrac{1}{q}}\).
This observation motivates us to formulate a duality gap between arbitrary primal and dual vectors \(u\in X\) and \(\zeta\in X^*\), which is zero if they are dual eigenvectors (under extra assumptions, this becomes an iff).
\begin{definition}[Duality gap]\label{def:duality_gap}
    We define the duality gap of the eigenproblem \(g\colon X\times X^*\to \R\) as 
    \[g(u,\zeta) \coloneqq R^{-\nicefrac{1}{p}}(u) - \operatorname{sign}(J^*(\zeta))| R_*(\zeta)|^{\nicefrac{1}{q}}\,.\]
\end{definition}
The boundedness of the  primal and dual RQ, as discussed in \Cref{lemma:boundedness}, translates to the boundedness of the duality gap as
\[\infty > g(u, \zeta)\geq -\lambda_*^{\nicefrac{-1}{p}}\quad \forall u\in X, \zeta\in X^*\,.\]
If \(J\) is absolutely \(p\)-homogeneous, then \(J^*\) and \(R_*\) are always nonnegative (see comment preceding \Cref{ex:optimal_norm_ineq}) and the duality gap is symmetrically bounded in the sense that
\[\lambda_*^{\nicefrac{-1}{p}} \geq g(u, \zeta)\geq -\lambda_*^{\nicefrac{-1}{p}}\quad \forall u\in X, \zeta\in X^*\,.\]
Clearly, we have \(g(u,\zeta)=0\) for primal dual eigenvector pairs. The converse  statement is true under the additional assumption \(\zeta\in\partial J(u)\).
This will be shown below in detail in \Cref{prop:roots_dgap}.

\subsection{Geometric characterization}
In the following we formalize the natural intuition that eigenvectors are not changing their direction but are only scaled by application of the respective operator.
Recall that \(H\) is the power of a norm in the vector space \(X\).
Without loss of generality, we say \(H(u) = \frac{1}{p}|u|_H^p\), and \(H^*(\zeta)  =\frac{1}{q}|\zeta|_{H^*}^q\). We observe that \(\langle \zeta, u\rangle \leq |u|_H|\zeta|_{H^*}\) with equality iff \(\zeta\in \partial H(u)\).
The duality mapping helps us to formulate geometric properties for elements between the primal and dual space.
\begin{definition}[Cosine angle]\label{def:cosine_angle}
    We define the cosine of the angle between the elements \(u\in X\) and \(\zeta\in X^*\) as the cosine similarity
    \[\operatorname{cosim}(u, \zeta) \coloneqq \frac{\langle\zeta, u\rangle}{|u|_H|\zeta|_{H^*}}\,.\]
    We say \(u\) and \(\zeta\) are orthogonal iff \(\operatorname{cosim}(u,\zeta) = 0\).
    Furthermore, we say that they are parallel iff \(\operatorname{cosim}(u,\zeta) = 1\), and that they are  anti-parallel  iff \(\operatorname{cosim}(u,\zeta) = -1\).\\
\end{definition}
It is well known that the gradient of a differentiable function is either zero or perpendicular to the level set of the function at this point.
For an eigenvector \(u\), the subdifferentials \(\partial J(u)\) and \(\partial H(u)\), with positive scaling by the eigenvalue, share the same element in the dual space.
We are able to deduce that for eigenvectors there exists an element in the subdifferential of \(\partial J\) perpendicular to the sublevel set of \(H\).  
The sublevel sets of \(H\) are given by \(L_u(H) = \{v\in X| \, H(v)\leq H(u)\} = B_{|u|_H}(0)\), which are balls with radius \(|u|_H\) centered at the origin.
\begin{figure}[!htbp]
    \centering
    \includegraphics[width=0.42\textwidth]{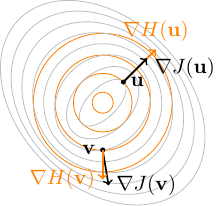}
    \caption{Sketch of the geometric characterization of nonlinear eigenvectors. The gray ellipses depict the level sets of \(J\) and the orange circles the level sets of \(H\), which are balls around the origin. There are two vectors \(u,v\)   with the gradients of \(H\) (orange) and \(J\) (black) drawn at these points. The gradients are perpendicular to the level sets. The vector \(v\) cannot be a an eigenvector since \(\nabla H(v)\) and \(\nabla J(v)\) do not share the same direction, thus \(\nabla J(v)\) is not normal to the balls around the origin. The gradients \(\nabla H(u)\) and \(\nabla J(u)\) at the point \(u\) share the same direction. Indeed, the vector \(u\) is an eigenvector if its norm fulfills the condition \(|u|_H = \frac{pJ(u)}{|\zeta|_{H^*}}\).
    }
    \label{fig:geom_charact}
\end{figure}

\begin{proposition}[Geometric characterization]\label{prop:geom_charact}
An element \(u \in X\) is a \(p\)-eigen\-vector of \(J\) with subgradient \(\zeta\) and eigenvalue \(\lambda\) iff
\begin{equation}\label{eq:geom_cond}
    \langle\zeta,u-v\rangle \geq 0~\forall v\in B_{|u|_H}(0), \quad |u|_{H} = \frac{pJ(u)}{|\zeta|_{H^*}}, \quad \text{ and  } \quad \zeta\in\partial J(u) \,.
\end{equation}
\end{proposition}
Before we give the proof, we interpret the result. Nonlinear eigenvectors are characterized as vectors on a sphere with radius chosen as their respective norm \(|u|_H\), for which the subgradient \(\zeta\) of \(J\) is normal to this sphere at \(u\) and has the length $ |\zeta|_{H^*} = \frac{pJ(u)}{|u|_{H}}$, see \Cref{fig:geom_charact} for a visualization.
Concerning the maximality in the variational inequality in \eqref{eq:geom_cond}  we observe 
\begin{align}\label{maximality}
\langle\zeta,u-v\rangle \geq 0~\forall v\in B_{|u|_H}(0) \quad \Rightarrow \quad
\langle \zeta, u\rangle = |\zeta|_{H^*} |u|_H\,,
\end{align}
as we will show in the proof below. This means that the aforementioned cosine similarity is 1 for \(\zeta\in\partial J(u)\), i.e., $u$ and $\zeta$ are parallel.
This angular condition easily reads as the well-known property of linear eigenvectors, namely, their direction is invariant under the application of the operator, but they only get scaled (positively in our case).

\begin{proof}[Proof of \Cref{prop:geom_charact}]
We start the proof by showing \eqref{maximality}.   To this end, by assumption we have
\begin{equation}\label{eq:geom_charact_inequal_proof}
    |\zeta|_{H^*}|u|_H \geq \langle \zeta, u\rangle \geq \langle \zeta, v\rangle \quad \forall v\in B_{|u|_H}(0)\,.
\end{equation}
We pick \(\tilde{v}\in \partial H^*\left(\frac{\zeta|u|_H^{p-1}}{|\zeta|_{H^*}}\right)\), which by (the dual version of) \eqref{newequation} satisfies
\[|\tilde{v}|_H = \left|\frac{\zeta|u|_H^{p-1}}{|\zeta|_{H^*}}\right|_{H^*}^{q-1} = |u|_H,\]
where we used   $(p-1)(q-1) = 1$.
Thus, \(\tilde{v}\in B_{|u|_H}(0)\) and   we get by \eqref{newequation}--\eqref{eq:euler_ident}
\[\langle \zeta, \tilde{v} \rangle = \frac{|\zeta|_{H^*}}{|u|_H^{p-1}}\langle \frac{\zeta|u|_H^{p-1}}{|\zeta|_{H^*}}, \tilde{v}\rangle = \frac{|\zeta|_{H^*}}{|u|_H^{p-1}} \left| \frac{\zeta|u|_H^{p-1}}{|\zeta|_{H^*}}\right|_{H^*}\left| \tilde{v}\right|_H = |\zeta|_{H^*}|\tilde{v}|_H = |\zeta|_{H^*}|u|_H.\]
With \eqref{eq:geom_charact_inequal_proof} we receive that \(\langle \zeta, u\rangle = |\zeta|_{H^*}|u|_H\).  

We now start the obtain proof with the implication that, if \(u\) is a \(p\)-eigenvector of \(J\) with subgradient \(\zeta\) and eigenvalue \(\lambda\), then the geometric characterization holds. Clearly, $\zeta \in \partial J(u)$. Choose  \(\eta\in\partial H(u)\) such that \(\zeta=\lambda\eta\) for \(\lambda = R(u) =\frac{pJ(u)}{|u|_H^p}\).
Then by the definition of the subdifferential it holds
\begin{align*}
\langle\eta,v-u\rangle + \frac{1}{p}|u|_H^p\leq\frac{1}{p}|v|_H^p\quad\forall v \in X \quad 
\Rightarrow \quad  \langle\eta,v-u\rangle \leq 0\quad\forall v \in B_{|u|_H}(0)\,.
\end{align*}
Since $\lambda>0$, the first property in \eqref{eq:geom_cond} holds. With \(\eta\in\partial H(u)\), and \eqref{newequation} we get
\[\lambda|u|_H^p = \lambda|u|_{H}^{p-1}|u|_H = \lambda|\eta|_{H^*}|u|_H = |\zeta|_{H^*} |u|_H\,. \]
Thus, as \(\lambda|u|_H^p = pJ(u)\), we have \(|u|_{H} = \frac{pJ(u)}{|\zeta|_{H^*}}\).

Now, for the converse implication, we consider \(\zeta\in\partial J(u)\) such that \eqref{eq:geom_cond} holds. We show that \(u\) is \(p\)-eigenvector with subgradient \(\zeta\) by verifying \(\frac{\zeta}{R(u)}\in \partial H(u)\), with \(R(u)\).
Indeed, by \eqref{eq:geom_cond}, \eqref{maximality} we get \(pJ(u) =|\zeta|_{H^*}|u|_H= \langle\zeta, u\rangle\). Thus, for all $v \in X$ by Cauchy-Schwarz and Young's inequality
\begin{align*}
    \langle \frac{\zeta}{R(u)}, v-u\rangle + \frac{1}{p}|u|_H^p &=\frac{\langle\zeta, v\rangle}{R(u)} - \frac{\langle\zeta, u\rangle}{R(u)} + \frac{1}{p}|u|^p_H \\
    &\leq \frac{|\zeta|_{H^*}|v|_H}{R(u)} - \frac{\langle \zeta, u\rangle}{R(u)} + \frac{1}{p}|u|_H^p\\
    &=\frac{|\zeta|_{H^*} |v|_H |u|^p_H}{pJ(u)} - \frac{\langle\zeta, u\rangle|u|^p_H}{pJ(u)} + \frac{1}{p}|u|^p_H \\
    &=\frac{|\zeta|_{H^*} |v|_H |u|^p_H}{|\zeta|_{H^*}|u|_H} - \frac{\langle\zeta, u\rangle|u|^p_H}{\langle \zeta, u\rangle} + \frac{1}{p}|u|^p_H \\
    &= |u|^{p-1}_H|v|_H - |u|_H^p + \frac{1}{p}|u|_H^p\\
    &\leq \frac{1}{q}|u|_H^p + \frac{1}{p}|v|^p_H - |u|_H^p + \frac{1}{p}|u|^p\\
    &= \frac{1}{p}|v|^p_H\,.
\end{align*}
This concludes the proof.
\end{proof}
\begin{remark}[Homogeneous case]\label{the homogeneous case}
For \(J\) absolutely \(p\)-homogeneous and \(\zeta\in\partial J(u)\), the cosine similarity is always nonnegative, due to the Euler identity \(0\leq pJ(u)= \langle\zeta,u\rangle\).
Additionally, the condition for the vector length \(|u|_{H} = \frac{pJ(u)}{|\zeta|_{H^*}}\) in \eqref{eq:geom_cond} is not necessary because it always holds due to the Euler identity and \eqref{maximality}.
Due to the homogeneity, it is not restrictive to consider only eigenvectors with norm 1. 
\end{remark}
\begin{remark}[Maximizing the Rayleigh quotient in the homogeneous case]
In \Cref{ex:max_rq}, we showed that maximizers of the RQ are eigenvectors if the primal RQ is bounded and  maximizers exist.  
    In the case of \(J\) being absolutely \(p\)-homogeneous, maximizing the RQ is equivalent to maximizing a convex function over a closed convex domain, namely the unit ball induced by \(H\). In fact,
    \begin{align}\label{forlater}
    \underset{u\in X}{\max\,}\frac{pJ(u)}{|u|_H^p} =
    \underset{u\in X}{\max\,}pJ\left(\frac{u}{|u|_H}\right) =
    \underset{u\in X,\, |u|_H=1}{\max\,}pJ(u) =
    \underset{u\in X,\, |u|_H\leq1}{\max\,}pJ(u)\,.
    \end{align}
    For convex and nonconstant \(J\), the maximum over the unit sphere is the maximum over the unit ball and it is attained at the extremal values of this set \cite{Rockafellar1970convana}.
    Rockafellar characterizes the subgradients of maximizers as the non-zero vectors normal to the convex set \cite{Rockafellar1970convana}.
    In particular, he deduces the eigenvalue condition \(\lambda u \in \partial J(u)\) and \(|u|=1\) in the case of the unit Euclidean ball as a convex set in finite dimensions \cite[Theorem 32.4]{Rockafellar1970convana}.
    
  A similar result holds in our general setting: for a vector \(\bar{u}\) attaining the maximum of \(pJ\) on the unit ball, we denote its subgradients as \(\bar{\zeta}\in\partial pJ(\bar{u})\).
    By \Cref{ex:max_rq} and \eqref{forlater} we know that \(\bar{u}\) is an eigenvector. From \Cref{prop:geom_charact} we hence get that \(\bar{\zeta}\) is an element normal to the unit ball with respect to \(|\cdot|_H\) at \(\bar{u}\), namely
    \begin{equation*}
    	\langle \bar{\zeta}, \bar{u} - u\rangle \geq 0,\quad \forall u \in B_1(0)\,.
    \end{equation*}
\end{remark}
We recall that by \Cref{thm:duality} primal and dual eigen\-vectors are equivalent via the subgradient relation with reciprocal eigenvalues defined via their respective RQ. In \Cref{prop:geom_charact}, we characterize primal and dual eigenvectors as being parallel in terms of an angle induced by \(H\).

In the following, we analyze the cosine similarity between primal and dual eigenpairs, which allows us to identify the zeros of the duality gap introduced in \Cref{def:duality_gap}.
\begin{lemma}[Connection of the primal and dual RQ]\label{lemma:RQlemma}
Let \(\zeta\in \partial J(u)\) and \(J^*(\zeta)\geq 0\).
Then, for the rescaled primal and dual Rayleigh quotients it holds that
\[R_*^{\nicefrac{1}{q}}(\zeta) \: \leq \: \operatorname{cosim}(u,\zeta)\,R^{-\nicefrac{1}{p}}(u)\,.\]
Equality implies that    $ \left(qJ^*(\zeta)\right)^{\nicefrac{1}{q}}\left(pJ(u)\right)^{\nicefrac{1}{p}} = J(u) + J^*(\zeta)$.
\end{lemma}
\begin{proof}
By using the Young inequality and the Fenchel-Young equality \eqref{eq:Fenchel_Young_eq} we rewrite the primal and dual RQ as
\begin{align*}
    R_*^{\nicefrac{1}{q}}(\zeta)\,R^{\nicefrac{1}{p}}(u) \ &= \ \left(\frac{qJ^*(\zeta)}{|\zeta|_{H^*}^q}\right)^{\nicefrac{1}{q}}\left(\frac{pJ(u)}{|u|_H^p}\right)^{\nicefrac{1}{p}} \\
    \ &\leq \  \frac{J(u) + J^*(\zeta)}{|u|_H|\zeta|_{H^*}} = \frac{\langle\zeta, u\rangle}{|u|_H|\zeta|_{H^*}}\tag{\theequation}\label{eq:rq_connection_ineq} = \operatorname{cosim}(u,\zeta)\,.
\end{align*}
As the Rayleigh quotients are nonnegative, the proof is concluded.
\end{proof}
This statement implies that the duality gap is always nonnegative for \(\zeta\in\partial J(u)\) and \(J^*(\zeta)\geq0\). For absolutely \(p\)-homogeneous functionals \(J\), by employing the Euler identity instead of the Young inequality, equality holds in \eqref{eq:rq_connection_ineq}.

\begin{proposition}[Roots of duality gap]\label{prop:roots_dgap}
The duality gap is zero for \(\zeta\in \partial J(u)\) if and only if \(u\) is a \(p\)-eigenvector of \(J\) with subgradient \(\zeta\) and eigenvalue \(\lambda\).
\end{proposition}
\begin{proof}
    The implication that the primal dual eigenvector pair is a root of the duality gap simply follows from the eigenvalue relation stated in \Cref{thm:duality}, where we use that the eigenvalues (and thus the RQ) are positive.
    
    For the opposite direction, we have
    \begin{align*}
        g(u,\zeta) = 0
        \ \Leftrightarrow \ R^{-\nicefrac{1}{p}}(u) = \operatorname{sign}(J^*(\zeta))| R_*(\zeta)|^{\nicefrac{1}{q}}\,.
    \end{align*}
    With \(R(u)>0\) we get \(\operatorname{sign}(J^*(\zeta)) = 1\) and \(| R_*(\zeta)|^{\nicefrac{1}{q}} = R_*^{\nicefrac{1}{q}}(\zeta)\).
    We can then apply \Cref{lemma:RQlemma} to obtain
    \begin{equation}\label{eq:gap}
            0 = g(u, \zeta) =  R^{-\nicefrac{1}{p}}(u) - R_*^{\nicefrac{1}{q}}(\zeta)\geq (1-\operatorname{cosim}(u,\zeta))\,R^{-\nicefrac{1}{p}}(u)\: \geq \: 0\,.
    \end{equation}
    From this chain of inequalities we deduce that the angular condition \(\operatorname{cosim}(u,\zeta)=1\) holds. To conclude, we apply \Cref{prop:geom_charact}.
    For this, in view of \eqref{maximality}, it remains to check the length condition in \eqref{eq:geom_cond}, namely \(pJ(u) = |\zeta|_{H^*}|u|_H\). We observe that $|\zeta|_{H^*}|u|_H = \langle \zeta, u\rangle = J(u) + J^*(\zeta)$
    by \eqref{eq:Fenchel_Young_eq},
i.e., it suffices to check \(pJ(u) =qJ^*(\zeta)\). By equality in \eqref{eq:gap} and \Cref{lemma:RQlemma} we get
    \begin{equation*}
        \left(qJ^*(\zeta)\right)^{\nicefrac{1}{q}}\left(pJ(u)\right)^{\nicefrac{1}{p}} = J(u) + J^*(\zeta)\,.
    \end{equation*}
 By  Young's inequality, this equality is in one-to-one correspondence with \(pJ(u) = qJ^*(\zeta)\). Thus, it holds \(|\zeta|_{H^*}|u|_H = J(u)+J^*(\zeta) = pJ(u)\).
\end{proof}
\subsection{Eigenproblem of the inverse operator}

Consider \(J\in \Gamma_0(X)\). In view of the equivalence \(\zeta\in\partial J(u)\Leftrightarrow u\in\partial J^*(\zeta)\), the subdifferential of the dual functional can be interpreted as the inverse of the subdifferential of the primal functional. 
For a partial differential operator, the inverse operator is often compact, consider e.g., the linear operator \(A = \frac{d}{dx}\), \(A\colon C^1([0,1])\to C^0([0,1])\). Its inverse operator is the integral operator \(A^{-1}\colon C^0([0,1])\to  C^0([0,1]) \) , \(A^{-1}f=\int_0^xf(y)\,dy\), which is compact  by Arzelà-Ascoli.
 We specifically highlight this fact  because in the next section we consider inverse methods in terms of the dual and pose compactness assumptions, which are often satisfied for partial differential operators, since we are working with the dual.
Throughout this subsection we assume the functional \(J\) to be differentiable to avoid the set-valued notation. For energy functionals whose first variation describes a partial differential operator, we will write
\[\zeta = \mathcal{D}(u) \coloneqq\partial J(u)\,.\]
Then, the resolving operator or inverse operator of the partial differential operator is given by
\[u = \mathcal{D}^{-1}(\zeta) = \partial J^*(\zeta)\,.\]
The eigenproblem of the resulting PDE can be described with the framework introduced in this work as
\begin{align*}
    \mathcal{D}(u) - \lambda \partial H(u)\ni0\,,
\end{align*}
where \(\partial H\) is a suitable duality map. 
Its dual counterpart is then given as
\[
\mathcal{D}^{-1}(\zeta)- \mu \partial H^*(\zeta)\ni0\,,
\]
which involves the inverse operator and the inverse duality map. It is equivalent to the primal problem according to \Cref{thm:duality}.
By the theory developed in \Cref{subsect:dual_problem} we observe that the dual problem is the eigenproblem of the inverse operator with \(\mu=\lambda^{1-q}\).
\begin{example}[\(p\)-Laplace]
In the example of the \(p\)-Laplace eigenproblem described in \Cref{ex:p-laplace}, with \(|\cdot|_H = \|\cdot\|_{L^p(\Omega)}\) and \(|\cdot|_{H^*}= \|\cdot\|_{L^q(\Omega)}\), we have the equivalent dual problem
    \[\left(-\Delta_p\right)^{-1}(\zeta) - \mu |\zeta|^{q-2}\zeta = 0 \]
    on the dual space \(X^* = W^{-1,q}(\Omega)\) with \(\mu = \lambda^{1-q}\).
\end{example}
\begin{example}[Matrix eigenproblem]
    The dual eigenproblem of the eigenproblem for symmetric positive definite matrices as in \Cref{ex:spd} is given by 
    \begin{equation*}
        A^{-1}\zeta = \frac{1}{\lambda}\zeta\,.
    \end{equation*}
\end{example}

\section{Gradient flows and inverse power methods}\label{sect:GFandIPM}

In this section, we connect two prominent approaches from the literature which approximate solutions to nonlinear eigenproblems, namely the inverse power method (IPM) and methods based on gradient flows.
In the following, we assume \(J \in \Gamma_0(X)\) to be absolutely \(p\)-homogeneous and \((X, |\cdot|_J)\) to be a reflexive Banach space compactly embedded into the reflexive Banach space \((X, |\cdot|_H)\). We refer to \Cref{ex:plpl_ipm}  below for an explicit example. 

First,  we introduce the inverse power method for nonlinear eigenproblems as defined in the finite dimensional case by Hein and Bühler \cite{Hein_Buehler} with application to spectral clustering and sparse PCA,  or applied in the infinite dimensional setting to the \(p\)-Laplacian by Bozorgnia and Arakelyan in \cite{Farid_B_plpl, Farid_B_2ndplpl}.
The nonlinear inverse power method was analyzed in reflexive Banach spaces for homogeneous functionals to solve an abstract eigenproblem by Ercole  \cite{Ercole2018abstract_NonLinEigValPb}. Given an initial vector \(u^0\in \operatorname{dom}(J)\), \(|u^0|_H=1\), the method can be formulated  as the  iteration scheme
\begin{align}\label{eq:ipm}\tag{IPM}
    \begin{split}
        \zeta^k &\in \partial H(u^k)\\
        u^{k+\nicefrac{1}{2}} &\in \partial J^*(\zeta^k)\\
        u^{k+1} &= \frac{u^{k+\nicefrac{1}{2}}}{|u^{k+\nicefrac{1}{2}}|_H}\,.
    \end{split}
\end{align}
Secondly, RQ minimizing flow-based methods in Hilbert spaces were investigated by Feld et al.\ \cite{Feld_2019}, Aujol et al.\ \cite{Aujol2018_onehomogflow}, Cohen and Gilboa \cite{COHEN2018Gilboa_flow}, and by Nossek and Gilboa \cite{NossekG16flowgenerating}. It has been proven by Bungert and Burger \cite{Bungert2022GFburger} that they are asymptotically equivalent to normalized gradient flows which are, for instance, given by  
\begin{equation}\label{eq:norm_gf}
    \begin{cases}
        v'(t) + \partial J(v(t))\ni 0, &v(0)=f,\\
        u(t) = \frac{v(t)}{|v(t)|_H}\,. &{}
    \end{cases}
\end{equation}
It can be shown that the long-time behavior of the normalized gradient flow is directly related to solutions of nonlinear eigenproblems \cite{Varvaruca2004evoequ, Bungert2022GFburger, Bungert2020asympt}. The flow \eqref{eq:norm_gf} can also be formulated as a doubly nonlinear gradient flow in Banach spaces by using the duality map to replace \(v'\) with \(\partial H(v')\).
The implicit Euler discretization yields
\begin{equation}\label{eq:norm_gf_ELEQ}
    \begin{cases}
        \partial H\left(\frac{u^{k+\nicefrac{1}{2}} - u^k}{\tilde{\tau}}\right) + \partial J\left(u^{k+\nicefrac{1}{2}}\right) \ni 0,\\
        u^{k+1} = \frac{u^{k+\nicefrac{1}{2}}}{|u^{k+\nicefrac{1}{2}}|_H}\,,
    \end{cases}
\end{equation}
where $\tilde{\tau}$ denotes the time step size.  
  The energy linked to the Euler-Lagrange equation \eqref{eq:norm_gf_ELEQ} defines the proximal power method  (PPM) \cite{ambrosioGradientFlowsMetric2008, Bungert2022GFburger}
 with an initial value \(u^0\in X\), \(|u^0|_H=1\) as
\begin{align}
    \begin{cases}
        u^{k+\nicefrac{1}{2}} \in \underset{v\in X}{\operatorname{arg\, min}}\, H(v-u^k) + \tilde{\tau}^{p-1}J(v)\,,\\
    u^{k+1} = \frac{u^{k+\nicefrac{1}{2}}}{|u^{k+\nicefrac{1}{2}}|_H}\,.
    \end{cases} \label{eq:p_prox_method}\tag{PPM}
    \quad 
\end{align}
The existence of the iterations in \((X, \|\cdot\|_X)\) is guaranteed under relatively mild assumptions.
It is enough to require the compactness assumption from above, namely \((X, |\cdot|_J)\) is compactly embedded into \((X, |\cdot|_H)\), and
the lower semi-continuity of \(J\) with respect to \(|\cdot|_H\), see \cite[Section 2.2]{ambrosioGradientFlowsMetric2008}.
We will assume these two properties throughout this section.

In  the sequel, we show that the inverse power method converges towards an eigenfunction, and monotonically increases the dual RQ.
The latter result even holds without homogeneity assumptions on \(J\).
Furthermore, we will show that the proximal power method falls under the broad category of inverse power methods.

\subsection{Convergence of the inverse power method}\label{subsection:conv_ipm}

We start with the monotonic behavior with respect to the dual RQ of the IPM.

\begin{lemma}[Monotonicity of the dual RQ]\label{lemma:monotonicity_dRQ}
    The inverse power method \eqref{eq:ipm} increases the dual Rayleigh quotient in the sense that
    \[R_*(\zeta^k)\leq R_*(\zeta^{k+1})\,.\]
\end{lemma}
\begin{proof}
For all iterates in \eqref{eq:ipm}, we have that \(|\zeta^k|_{H^*}=|u^k|^{p-1}_H=1\), see \eqref{newequation}.
    Therefore, we just need to show that \(J^*(\zeta^{k+1})\geq J^*(\zeta^k)\).
    By definition of the subdifferential we have
    \[\langle\zeta^{k+1} - \zeta^k, u^{k+\nicefrac{1}{2}}\rangle + J^*(\zeta^k)\leq J^*(\zeta^{k+1})\,.\]
    Thus, it is sufficient to show \(\langle\zeta^{k+1} - \zeta^k, u^{k+\nicefrac{1}{2}}\rangle\geq 0\).
    We note that
    \[\langle \zeta^k, u^{k+\nicefrac{1}{2}}\rangle \leq |\zeta^k|_{H^*}|u^{k+\nicefrac{1}{2}}|_H = |u^{k+\nicefrac{1}{2}}|_H\,.\]
    Again using  \eqref{eq:ipm} we get $|u^{k+\nicefrac{1}{2}}|_H^{p-1}\zeta^{k+1}   \in  \partial H(u^{k+\nicefrac{1}{2}})$ and then \eqref{eq:euler_ident} implies 
        \begin{align*}
        \langle\zeta^{k+1}, u^{k+\nicefrac{1}{2}}\rangle &= \frac{1}{|u^{k+\nicefrac{1}{2}}|_H^{p-1}} \langle |u^{k+\nicefrac{1}{2}}|_H^{p-1} \zeta^{k+1}, u^{k+\nicefrac{1}{2}}\rangle =          \frac{|u^{k+\nicefrac{1}{2}}|_H^p}{|u^{k+\nicefrac{1}{2}}|_H^{p-1}} = |u^{k+\nicefrac{1}{2}}|_H\,.
    \end{align*}
    This concludes the proof.
\end{proof}
Note that the homogeneity of \(J\) was not needed to prove the monotonicity of the dual RQ. Next, we
 analyze the convergence of the inverse power method on \((X, |\cdot|_{H})\) and \((X, |\cdot|_{J})\). For this purpose, we note that,  due to the reflexivity of the spaces, the compact embedding of \((X, |\cdot|_J)\) into \((X, |\cdot|_H)\) translates  to the compact embedding of \((X^*, |\cdot|_{H^*})\) into \((X^*|\cdot|_{J^*})\).   
 
\begin{theorem}[Convergence of the inverse power method]\label{thm:convergence_ipm}
    There exists a sub\-sequence \(( u^{k_l})_{l\in\N} \) of the inverse power method \eqref{eq:ipm}, which converges strongly in \(|\cdot|_H\) and \(|\cdot|_J\) towards a \(p\)-eigenvector of \(J\) and the duality map \(\partial H\) in the sense of  \Cref{def:eigenproblem}.
\end{theorem}

\begin{proof}
    By the iterative scheme \eqref{eq:ipm} we obtain the sequences \((u^k)_{k\in\N}\subset X\), \((u^{k + \nicefrac{1}{2}})_{k\in\N}\subset X\), and \((\zeta^k)_{k\in\N}\subset X^*\) with their relation \(\zeta^k\in \partial H(u^k)\).
    For the sake of brevity, we further define the sequence \(v^k\coloneqq u^{k+\nicefrac{1}{2}}\).
    From this it follows that \(u^{k+1} = \frac{v^k}{|v^k|_H}\) and \(v^k \in \partial J^*(\zeta^k)\) .\\
    The proof is performed in two steps.
    First, we show the convergence of the iterates in the two relevant norms.
    In a second step, we show that the limit satisfies the nonlinear eigenvector condition.

\emph{Step 1.} The convergence is proven with a standard argument in which we show boundedness of the sequences and then, with a compactness argument, we obtain the convergence of a subsequence, which we do not relabel for notational convenience.\\
    We remind ourselves that it holds \(|\zeta|_{J^*}^{q-1} = |v|_J\) and \(| u|_H = |\zeta|_{H^*}^{q-1}\) for \(v\in \partial J^*(\zeta)\) and \(\zeta\in \partial H(u)\).
    This will be our key idea to translate boundedness between the spaces.
    We observe from the scheme \eqref{eq:ipm} that \(|u^k|_H = 1 = |\zeta^k|_{H^*}\). Thus, the sequences \((u^k)_{k\in\N}\subset X\) and \((\zeta^k)_{k\in\N}\subset X^*\) are bounded in \((X, |\cdot|_H)\) and \((X^*,|\cdot|_{H^*})\), respectively.

    With the compactness of the embedding there exists a subsequence of \( (\zeta^k)_{k\in\N}\) which converges to some $ \bar{\zeta}$ strongly in \((X^*, |\cdot|_{J^*})\) and weakly in \((X^*, |\cdot|_{H^*})\).
    Furthermore, \((\zeta^k)_{k\in\N}\subset X^*\) is bounded in \((X^*,|\cdot|_{J^*})\).
    With \(|v^k|_J = |\zeta^k|_{J^*}^{q-1}\) we easily see that \((v^k)_{k\in\N}\) is bounded in \((X,|\cdot|_J)\).
    We once again use the compact embedding and deduce that \((v^k)_{k\in\N}\) is bounded in \((X,|\cdot|_H)\) and has a converging subsequence in \((X,|\cdot|_H)\).
   Recalling \(u^{k+1} = \frac{v^k}{|v^k|_H}\), we thus see that \((u^k)_{k\in\N}\) has a converging subsequence in \((X,|\cdot|_H)\) whose limit we denote by $\bar{u}$.
   
    It remains to prove that \((u^k)_{k\in\N}\) is bounded and has a subsequence which converges to $\bar{u}$ also in \((X,|\cdot|_J)\). As a byproduct, we also check  that \((\zeta^k)_{k\in\N}\) has a  subsequence converging to $\bar{\zeta}$ in \((X^*,|\cdot|_{H^*})\).
    To do so, we will show that the convergence of \((u^k)_{k\in\N}\) in \((X,|\cdot|_H)\) and the boundedness of \((\zeta^k)_{k\in\N}\) in \((X^*,|\cdot|_{H^*})\) with \(\zeta^k\in\partial H(u^k)\) implies the convergence of \((\zeta^k)_{k\in\N}\) in \((X^*,|\cdot|_{H^*})\).
    For this, we  use the definition of the subdifferential
    \[H(u^k) - \langle\zeta^k,u^k\rangle \leq H(v) - \langle \zeta^k, v\rangle \quad \forall v\in X\,.\]
Passing to the limit, we see that \(\bar{\zeta}\in\partial H(\bar{u})\) because of the strong convergence of \((u^k)_{k\in\N}\) to $\bar{u}$ in \((X,|\cdot|_{H})\) and the weak convergence  of \((\zeta^k)_{k\in\N}\) to $\bar{\zeta}$  in \((X^*,|\cdot|_{H^*})\).
The property \(\bar{\zeta}\in\partial H(\bar{u})\) implies  $|\bar{\zeta}|_{H^*} = |\bar{u}|_H^{p-1}$. Together with $|u^k|^{p-1}_H =  |\zeta^k|_{H^*}$ for all $k \in \N$ and $|\bar{u}|_H = \lim_{k\to\infty}|u^k|_H$, we find $|\bar{\zeta}|_{H^*} = \lim_{k\to\infty}|\zeta^k|_{H^*}$. This along with the weak convergence  of \((\zeta^k)_{k\in\N}\)   to $\bar{\zeta}$  implies that \((\zeta^k)_{k\in\N}\)  converges to $\bar{\zeta}$ strongly in \((X^*,|\cdot|_{H^*})\).     
  
  We now use the same argument to show convergence of a subsequence of \((v^k)_{k\in\N}\) in \((X, |\cdot|_J)\) to some $\bar{v}$. In fact, we know that \((v^k)_{k\in\N}\) is bounded in  \((X, |\cdot|_J)\) (and thus weakly converging, up to a further subsequence) and \(v^k\in\partial J^*(\zeta^k)\), with \((\zeta^k)_{k\in\N}\) strongly converging in  \((X^*, |\cdot|_{J^*})\).
    Passing the limit in
  \[J^*(\zeta^k)-\langle\zeta^k, v^k\rangle\leq J^*(\eta)-\langle\eta, v^k\rangle\quad \forall \eta\in X^*\]
  yields  $\bar{v} \in \partial J^*(\bar{\zeta})$ thus \(|\bar{\zeta}|_{J^*}=|\bar{v}|_J^{p-1}\). With \(|\bar{\zeta}|_{J^*} = \lim_{k\to\infty}|\zeta^k|_{J^*}\) and \(|\zeta^k|_{J^*} = |v^k|_J^{p-1}\) for all \(k\in\N\), we find \(|\bar{v}|_J = \lim_{k\to\infty}|v^k|_J\). This along with weak convergence provides the strong convergence in  \((X, |\cdot|_J)\) of $v^k$ to $\bar{v}$.
 Eventually, the strong convergence of \((v^k)_{k\in\N}\) in \((X, |\cdot|_J)\) together with the definition \(u^{k+1} = \frac{v^k}{|v^k|_H}\)  implies that also \((u^k)_{k\in\N} \) converges strongly in \((X, |\cdot|_J)\), with limit $\bar{u} = \frac{\bar{v}}{|\bar{v}|_H}$.   Note that, as a byproduct of the above argument, we get $\bar{v} \in \partial J^*(\bar{\zeta})$.

\emph{Step 2.} We now prove that the limits fulfill the nonlinear eigenvector equation \eqref{eq:eigenproblem}.
    Due to the strong convergence of \((\zeta^k)_{k\in\N}\) in both topologies, we have that \(R_*(\zeta^k)\to R_*(\bar{\zeta)}\).
We observe that \(R_*(\bar{\zeta)} = \frac{qJ^*(\bar{\zeta)}}{|\bar{\zeta}|_{H^*}^q} = \frac{\langle\bar{\zeta}, \bar{v}\rangle}{1}\), due to the Euler identity \eqref{eq:euler_ident}, $\bar{v} \in \partial J^*(\bar{\zeta})$, and \(|\bar{\zeta}|_{H^*}= \lim_{k \to \infty}|\zeta^k|_{H^*} = 1\).
Furthermore, it holds that \(\langle\bar{\zeta}, \bar{v}\rangle = \langle\bar{\zeta},|\bar{v}|_H \bar{u}\rangle =  |\bar{v}|_H\) by \(\bar{u} = \frac{\bar{v}}{|\bar{v}|_H}\) and \(\langle\bar{\zeta}, \bar{u}\rangle = |\bar{u}|^p_H = 1 \), where we use that \(\bar{\zeta}\in\partial H(\bar{u})\).
We therefore have
\begin{align*}
 R_*(\bar{\zeta}) =  |\bar{v}|_H \,.
 \end{align*}
Hence, setting $u := \bar{u}$ and $\zeta := |\bar{v}|_H^{1-p} \bar{\zeta}$ and by  recalling that  $\bar{v} \in \partial J^*(\bar{\zeta}) \Leftrightarrow \bar{\zeta} \in \partial J(\bar{v})$ and \(\bar{\zeta}\in\partial H(\bar{u})\), we conclude that     \({\zeta}\in\partial J({u})\) and that $\zeta$ fulfills  the  eigenvector condition $ \zeta \in \lambda\partial H(u)$ with $\lambda := |\bar{v}|_H^{1-p}$. It remains to show that $\lambda = R(u)$.
 With the Euler identity \eqref{eq:euler_ident} we have \(pJ(u) = \langle\zeta, u\rangle\) and \(pH(u) = \langle\frac{\zeta}{\lambda}, u\rangle \), and thus \[R(u) = \frac{J(u)}{H(u)} = \frac{\langle\zeta, u\rangle}{\langle \frac{\zeta}{\lambda}, u\rangle} = \lambda\,.\]
This concludes the proof.
\end{proof}
\begin{remark}\label{remark:groundstate_sign}
    This result does not guarantee that the nonlinear eigenfunction approximated with the inverse power method is the ground-state eigenfunction.
    In practice, to guarantee convergence towards the ground-state eigenfunction, one can proceed as follows.
    First, one has to look at the special structure of the ground state, e.g., the ground state of the  \(p\)-Laplacian is the only non-sign changing eigenfunction.
    Then, one has to check whether the operator \(\partial H\) and \(\partial J^*\) are preserving this structure.
   Note, however, that this approach provides sufficient criteria for the convergence towards ground states, which are in general not necessary, as will be shown in the numerical examples, where we compute the \(p\)-Laplacian ground state from a sign-changing initial guess.
\end{remark}

\begin{example}\label{ex:plpl_ipm}
Let us come back to the initial   \Cref{ex:p-laplace} of the \(p\)-Laplacian in the space \(W_0^{1,p}(\Omega)\) for a bounded domain \(\Omega\subset\R^n\), and let us apply \Cref{thm:convergence_ipm}.
    First, we need to check the assumptions in this setting.
    We have the uniformly convex and smooth reflexive Banach space \(X=W_0^{1,p}(\Omega)\) with its dual space \(X^* = W^{-1,q}(\Omega)\).
    The functionals are given as the power of norms, i.e., \(J(u) = \frac{1}{p}\|\nabla u\|_{L^p(\Omega)}^p\) and \(H(u) = \frac{1}{p}\|u\|_{L^p(\Omega)}^p\).
    The dual norm \(|\cdot|_{H^*}\) is thus given by the \(L^q\) norm. 
    The assumptions require \(W^{1,p}_0(\Omega)\) to be compactly embedded in \(L^p(\Omega)\), which is true due to the Rellich-Kondrachov theorem.  
    The inverse power method is thus applicable to the \(p\)-Laplace eigenproblem \eqref{eq:p-laplace} and converges towards an eigenfunction.
\end{example}

\subsection{Proximal power method as inverse power method}
We define the update step in \eqref{eq:p_prox_method} as the proximal operator with \(\tau := \tilde{\tau}^{p-1}\)
\begin{align}\label{proxop}
P_\tau(u) \coloneqq \underset{v\in X}{\operatorname{arg\, min}}\, H(v-u) + \tau J(v)\,.
\end{align}
The proximal operator is the solution to the minimization problem of the Moreau envelope of \(J\), given by
\[J_\tau(u) \coloneqq \underset{v\in X}{\operatorname{inf}}\, \frac{1}{\tau}H(v-u) + J(v)\,.\]
We note that this infimum is unique due to the strict convexity of \(H\) and \(J\) on \(X\).
We remind ourselves that we are working with powers of norms and have set without further restrictions \(X  =X\setminus\{\mathcal{N}(J)\cup \mathcal{N}(H)\}\).
The related Yosida approximation\footnote{For \(\tau \to 0\) the Yosida approximation is an element of \(\partial J\) with smallest \(|\cdot|_{H^*}\) norm in the sense \(|A_{ 0}(u)|_{H^*}\leq |\zeta|_{H^*} \) for all \(\zeta\in\partial J(u)\).} \cite{evans10} is defined as
\[A_\tau(u) \coloneqq \frac{\partial H\left(u - P_\tau(u)\right)}{\tau}\,.\]
The Yosida approximation gives us the connection of the proximal power method with the inverse power method.
The proximal operator step is an inverse iteration with duality mapping \(A_\tau\).

\begin{proposition}[Proximal step as inverse iteration]\label{prop:prox_step_inv_it}
    For \(u\in X\) and \(\tau>0\), there exists \(\eta\in A_\tau(u)\) such that
    \begin{align*}
        P_\tau(u) &\in \partial J^*(\eta)
    \end{align*}
    and for the same \(\eta\in A_\tau(u)\) it holds that
    \[\eta\in \partial J_\tau(u)\,.\]
\end{proposition}
\begin{proof}
By the definition of the proximal operator in \eqref{proxop} it holds that the Minkowski sum of the subdifferentials share a common element in the sense
\begin{align*}
    \partial H(P_\tau(u) - u) + \tau\partial J(P_\tau(u))&\ni0\\
    \Leftrightarrow -A_\tau(u) + \partial J(P_\tau(u))&\ni 0 \\        \Leftrightarrow A_\tau(u) \cap \partial J(P_\tau(u))&\neq \emptyset .
\end{align*}
We now choose \(\eta\in A_\tau(u)\) such that \(\partial J(P_\tau(u))\ni\eta\ \).

Let us prove that \(\eta\in \partial J_\tau(u)\). Note that  \(J_\tau\) lies in \(\Gamma_0(X)\) and that it is absolutely \(p\)-homogeneous, due to the definitions of \(J\) and \(H\).
    Its dual functional is given by
   \begin{align}\label{jstern}
   J_\tau^*(\zeta) = \tau^{q-1}H^*(\zeta) + J^*(\zeta)\,,
   \end{align}
    because the infimal convolution is the dual operation to the addition \cite[Theorem 16.4]{Rockafellar1970convana}.
Here, we also used $(\tau^{-1} H)^* = \tau^{q-1} H^*$ due to the $q$-homogeneity of $H^*$.
We deduce that the subdifferential of the dual functional is given by
   \begin{align}\label{jstern2}
   \partial J^*_\tau(\zeta) = \tau^{q-1}\partial H^*(\zeta) + \partial J^*(\zeta)\,.
   \end{align}
We now insert the  \(\eta\in A_\tau(u)\) chosen above (which satisfies \(\partial J(P_\tau(u))\ni\eta\) and thus equivalently \(P_\tau(u)\in\partial J^*(\eta)\)) into the dual functional \(\partial J^*_\tau(\eta)\). With the homogeneity of \(\partial H\) we get
\begin{align*}
    \partial J^*_\tau(\eta) &= \tau^{q-1}\partial H^*(\eta) + \partial J^*(\eta)\\
    &= \partial H^*(\tau\eta) + \partial J^*(\eta)\ni u-P_\tau(u) + P_\tau(u) = u\,,
\end{align*}
where we used that \(\partial H(u-P_\tau(u)) = \tau A_\tau(u) \ni \tau \eta\). This concludes the proof.
\end{proof}

\begin{remark}
    In the case of \(J\) and \(H\) being differentiable, the subdifferentials are single valued and thus we have \(A_\tau(u) = \partial J_\tau(u)\), called the Moreau theorem \cite[Proposition IV.1.8]{moreau_theorem}.
\end{remark}

Using the definition of $P_\tau$, the proximal power method as defined above in \eqref{eq:p_prox_method} can be written as
\begin{align*}
    u^{k+\nicefrac{1}{2}} &= P_\tau(u^k)\\
    u^{k+1} &=\frac{u^{k+\nicefrac{1}{2}}}{|u^{k+\nicefrac{1}{2}}|_H}\,.
\end{align*}
Thus, using  the inclusions  \(\eta\in \partial J_\tau(u)\) and  \(P_\tau(u) \in \partial J^*(\eta)\) given by  \Cref{prop:prox_step_inv_it}, we can write it as the inverse power method with dual mapping \(\partial J_\tau\) in place of \(\partial H\), namely
\begin{align*}\label{eq:prox_ipm}\tag{prox-ipm}
\begin{split}
    \zeta^k &\in \partial J_\tau(u^k)\\
    u^{k+\nicefrac{1}{2}} &\in \partial J^*(\zeta^k)\\
    u^{k+1} &=\frac{u^{k+\nicefrac{1}{2}}}{|u^{k+\nicefrac{1}{2}}|_H}.
\end{split}
\end{align*}
Strictly speaking,  \(|\cdot|_H\) should be replaced with \(|\cdot|_{J_\tau}\), but these two norms are equivalent on \(X\), due to the growth assumption \eqref{eq:growth_coercivity} on \(J\) and \(H\).
With the homogeneity of \(J\), \(H\), and \( J_\tau\), the iterates of \eqref{eq:prox_ipm} are equivalent up to scaling for both normalizations \(|\cdot|_H\) and \(|\cdot|_{J_\tau}\).
Thus, \eqref{eq:prox_ipm} describes the inverse power method, which solves the eigenproblem
\begin{equation}\label{eq:ppm_eigenproblem}
    \partial J(u) - \lambda_\tau \partial J_\tau(u)\ni 0\,.
\end{equation}
More precisely, as discussed in \Cref{subsection:conv_ipm}, this method  increases the dual RQ
\begin{equation}\label{eq:ppm_drq}
    R^\tau_*(\zeta) := \frac{J^*(\zeta)}{J^*_\tau(\zeta)}
\end{equation}
monotonously until we reach a steady state, which solves the nonlinear eigen\-problem.
To conclude this discussion, we check that the eigenproblem  \eqref{eq:ppm_eigenproblem}  for \(\tau>0\) is equivalent to the primal eigenproblem in \Cref{def:eigenproblem}. This is achieved by analyzing the dual RQ
of \Cref{def:dRQ} and \eqref{eq:ppm_drq}.
Bungert and Burger already proved this by characterizing eigenvectors as scaled fixed points of the proximal operator \cite[Theorem 5.2]{Bungert2022GFburger}.
We will present an alternative proof via the dual problem.
The eigenproblem solved by the proximal power method is given by
\[\partial J(u) - \lambda_\tau \partial J_\tau(u) \ni 0,\]
 with primal RQ \(\lambda_\tau = R_\tau(u) = \frac{J(u)}{J_\tau(u)}\). By using \eqref{jstern2} and \Cref{thm:duality}, its equivalent dual problem with dual eigenvalue \(\mu_\tau = \lambda_\tau^{1-q}\) reads  as
 \begin{align*}
 \partial J^*(\zeta) - \mu_\tau\partial J_\tau^*(\zeta) &\ni 0\\
    \Leftrightarrow \partial J^*(\zeta) - \mu_\tau\left(\tau^{q-1}\partial H^*(\zeta) + \partial J^*(\zeta)\right) &\ni 0\\
        \Leftrightarrow \partial J^*(\zeta) - \frac{\mu_\tau\tau^{q-1}}{1-\mu_\tau}\partial H^*(\zeta) &\ni 0.
\end{align*}
This shows that the dual eigenproblem of \eqref{eq:ppm_eigenproblem} is equivalent to the dual eigenproblem of \Cref{def:eigenproblem} with \(\mu = \frac{\mu_\tau\tau^{q-1}}{1-\mu_\tau}\).
In this context, we note that the dual RQ \eqref{eq:ppm_drq} is bounded from above by $1$, and thus \(\mu_\tau<1\), i.e.,  division by zero is impossible. We verify this by using the dual formulation of the growth assumption \eqref{eq:growth_coercivity} given by \( J^*(\zeta) \leq \lambda_*^{1-q}H^*(\zeta)\), see \Cref{lemma:boundedness}.
Recalling \eqref{jstern}, for \eqref{eq:ppm_drq} we indeed get \begin{equation*}
     R_*^\tau(\zeta) = \frac{J^*(\zeta)}{\tau^{q-1}H^*(\zeta) + J^*(\zeta)} \leq \frac{J^*(\zeta)}{((\tau\lambda_*)^{q-1}+1)J^*(\zeta)} = \frac{1}{(\tau\lambda_*)^{q-1}+1} \: < \: 1.
 \end{equation*}
Transferring the result of the equivalence of the two dual problems to the primal problems, by \Cref{thm:duality} we get  the equivalence of the eigenproblem \eqref{eq:ppm_eigenproblem} (which gets solved by the \eqref{eq:p_prox_method}) and \Cref{def:eigenproblem} with eigenvalue
$$\lambda 
= \frac{\lambda_\tau}{\tau}\left(  1-\lambda_\tau^{1-q}   \right)^{p-1}.$$

This is consistent with the formulation of Bungert and Burger \cite[Theorem 5.2]{Bungert2022GFburger}, who characterized fixed points of the proximal operator in the sense \(\lambda_\tau^{1-q}u=P_\tau(u)\) as eigenvectors with eigenvalue \(\lambda\).

\begin{remark}
    The Moreau envelope of \(J\) is the infimal convolution of \(\frac{1}{\tau}H\) and \(J\) and sometimes called the Moreau regularization of \(J\).
    With a similar argument as above we see the regulatory effect of the Moreau regularization on \(J\) with the aid of the nonlinear eigenvalues.
    The eigenvectors of \(J\) with Rayleigh quotient \(R(u) = \frac{J(u)}{H(u)}\) and \(J_\tau\) with Rayleigh quotient \(R_\tau(u)  =\frac{J_\tau(u)}{H(u)}\) are the same, but the eigenvalues of the regularization \(J_\tau\) get scaled as
    \[\lambda_\tau = \frac{\lambda}{(1+(\tau\lambda)^{q-1})^{p-1}}\,.\]
    Thus, they are asymptotically bounded from above by \(\frac{1}{\tau}\) and it holds that \(\lambda_\tau<\lambda\).
    For \(\lambda \ll \frac{1}{\tau}\) we have \(\lambda_\tau\approx\lambda\) and for \(\lambda \gg\frac{1}{\tau}\) we have \(\lambda_\tau \approx\frac{1}{\tau}\).
    A rescaling of the eigenvalues due to the application of a regularization is well known in linear inverse problems \cite{Kirsch_Intro_IP}, for example in Lavrentiev or Tikhonov regularization.
    We see that a rescaling happens as well for the nonlinear eigenvalues due to the Moreau regularization.
\end{remark}

\section{Numerical experiments}\label{sect:Numeric}
In this section, we numerically validate the convergence of the inverse power method introduced in \Cref{sect:GFandIPM} with the example of the \(p\)-Laplace eigenproblem.
Furthermore, we discuss approaches to compute higher-order eigenfunctions.
The python scripts for the implementation can be found on \href{https://github.com/JonaLbm/Duality_Perspective_on_Nonlinear_Eigenproblems_Computation_of_Eigenfunctions}{GitHub} \cite{github_repo}.

\subsection{Numerical validation of the inverse power method}\label{subsect:num_val_ipm}
The inverse power method \eqref{eq:ipm} is numerically implemented by choosing the functional \(J\) as the \(p\)-Dirichlet energy and the functional \(H\) as the \(p\)-th power of the \(L^p\)-norm.
The resulting pseudo code can be seen in \Cref{alg:IPM}.
\begin{algorithm}
\caption{Inverse power method for \(p\)-Laplacian}\label{alg:IPM}
\begin{algorithmic}
\Require $u^0$
\Ensure $\|u^0\|_{L^p} = 1$
\State $k =0$
\While{$k <30$}
    \State $\zeta^k \gets |u^k|^{p-2}u^k$ \Comment{Map \(u^k\) to the dual space}
    \State solve $-\Delta_p(u^{k+\nicefrac{1}{2}}) = \zeta^k$ \Comment{Solve the inner \(p\)-Dirichlet problem}
    \State  \(u^{k+1} \gets \frac{u^{k+\nicefrac{1}{2}}}{\|u^{k+\nicefrac{1}{2}}\|_{L^p}}\) \Comment{Normalize the iterate}
    \State \(k\gets k+1\)
\EndWhile
\end{algorithmic}
\end{algorithm}

We consider an L-shaped domain $(-1,1)^2 \setminus (0,1)^2$ and use a spatial discretization  within the square $(-1,1)^2$ of size  \(h=0.025\).
The \(p\)-Laplace operator is approximated  using a mean value formula \cite{Teso2021meanvalue}, and discretized with a finite differences scheme by del Teso and Lindgren \cite{delTeso2021}, which is similar to the \(p\)-Laplace formulation on graphs \cite{Daniel_Graph}.
The discretized two-dimensional \(p\)-Laplace operator is given by 
\begin{equation}\label{eq:disc p-lapl}
    \Delta_p^{h}(u(x)) = \frac{h^2}{D_{2,p}\pi r^{p+2}}\sum_{y\in B_r (x)} |(u(y) - u(x)|^{p-2}(u(y) - u(x)), \,
\end{equation}
where \(D_{2,p}\) is a constant, \(r\) the mean value radius, and \(h\) the spatial discretization. 
This provides a consist discretization if
\begin{equation}
    \label{eq:consistency} h = \begin{cases}
    o(r^{p/(p-1)}) &p\in (1, 3) \,,\\o( r^\frac{3}{2})& p \ge 3 \,,\\
\end{cases}
\end{equation}
 for \(r\to 0^+\).  
Del Teso and Lindgren approximate the Jacobian with the approxi\-mation of the \(p\)-Laplace operator \eqref{eq:disc p-lapl} to solve the \(p\)-Dirichlet problem using Newton's method in one dimension \cite{delTeso2021}.

We extend this numerical method from the one- to the two-dimensional case in order to have a robust solver for our recurrent \(p\)-Dirichlet problem.
We set the radius of the mean value approximation to \(r=0.2\).
We start with an initial guess of
\begin{equation}\tag{ex 1}\label{eq:num_ex_1}
    u^0(x_1, x_2) = -(1-2|x_1+0.5|)(1-2|x_2+0.5|)(1-|x_1|)(1-|x_2|)\,,
\end{equation}
see \Cref{fig:ipm_gs} top left, and initialize Newton's method to solve the inner problem with the current iterate \(u^k\).
We compute a maximum of $500$ Newton iterations or stop earlier if the solution solves the inner problem with an absolute error of \(10^{-12}\).
\begin{figure}[!htbp]
    \centering
    \includegraphics[width=\textwidth]{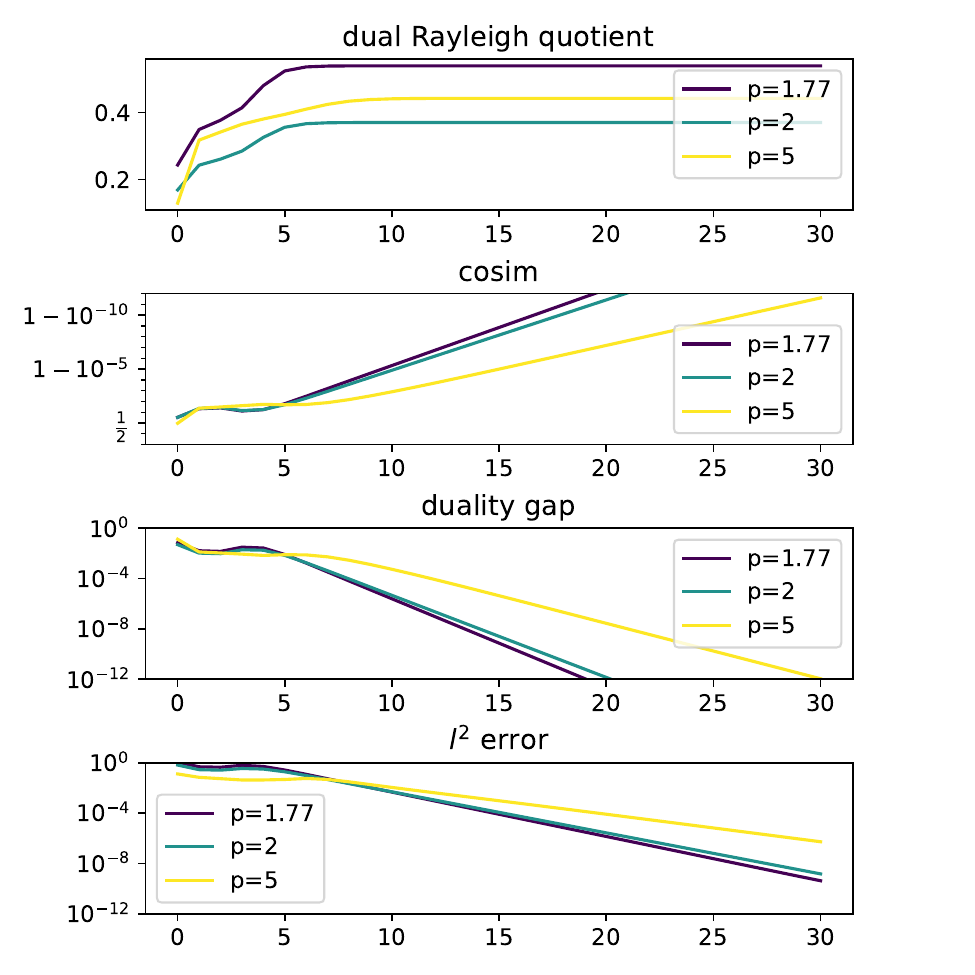}
    \caption{Convergence plots of the metrics of the inverse power method for different values of \(p\).
    The \(x\)-axis shows the number of iterations and the \(y\)-axis depicts the respective value of the metric.     We remark that the lower three plots use a logarithmic \(y\)-axis.
    From top  to down:  The dual Rayleigh quotient is shown to be monotonously increasing with respect to the iterates of the inverse power method.  $\operatorname{Cosim}$ is the cosine similarity, which becomes $1$ iff the iterate is a \(p\)-Laplace eigenvector, else it is less than one.
    Next, we illustrate the duality gap, which becomes zero if the iterate is an eigenvector and otherwise positive. Finally, we plot the \(l^2\)-error of the iterates fulfilling the eigenvector condition, which is zero iff the iterate is an eigenvector.    }
    \label{fig:ipm_conv}
\end{figure}
We perform $30$ steps of the inverse power method.
\begin{figure}[!htbp]
    \centering
    \includegraphics[width=\textwidth]{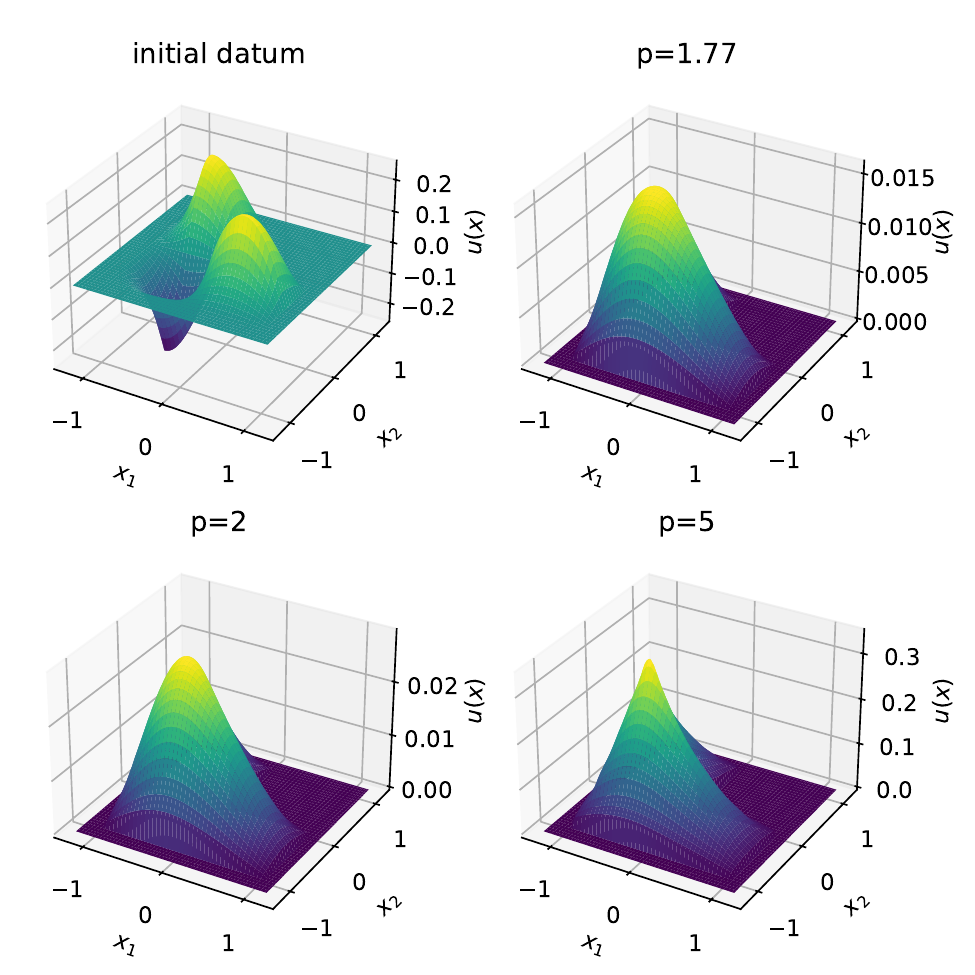}
    \caption{Initial guess (top left) and computed eigenvector for different values of \(p\), which is the ground state of the \(p\)-Laplacian on an L-shaped domain.
    }
    \label{fig:ipm_gs}
\end{figure}
In \Cref{fig:ipm_conv} we validate the behavior of the inverse power method and observe that it maximizes the dual RQ monotonously, as proven in \Cref{lemma:monotonicity_dRQ}, and converges towards an eigenfunction of the primal problem, cf.\ Theorem \ref{thm:convergence_ipm}.
The latter is verified by the behavior of our introduced metrics, i.e., the cosine similarity converging to one and the duality gap to zero.
The shown error depicts the \(l^2\)-error of the iterates inserted into the eigenproblem in the sense  \(\|\frac{ -\Delta_p(u^k)}{\|-\Delta_p(u^k)\|_q} - \frac{|u^k|^{p-2}(u^k)}{\||u^k|^{p-2}(u^k)\|_q}\|_2\).

It is important to observe that, in contrast to the monotone behavior of the dual RQ, the cosine similarity and  the duality gap are not optimized monotonously. Thus, we can state that the scheme does not converge towards the `closest' eigenfunction in the sense of the cosine similarity or the duality gap, but it maximizes the dual RQ until the scheme reaches a steady state.
We  observe in \Cref{fig:ipm_conv} and \Cref{fig:ipm_gs} that the inverse power method converges towards the ground state because it is the only nonsign-changing eigenfunction of the \(p\)-Laplacian \cite{lindqvist2008nonlinear}. This occurs even though the initial guess is changing its sign,   see \Cref{remark:groundstate_sign}.
We   observe   the characteristic change in steepness of the eigenfunction by varying the value of \(p\).

\subsection{Computation of higher-order eigenfunctions}
In the following we discuss approaches to compute higher-order eigenfunctions.
Our numerical examples for the computation  follow \cite[Example 7.3.]{bobkov2025inverseiterationmethodhigher}, i.e., we set \(p=3\) and use the initial guess 
\begin{equation}\tag{ex 2}\label{eq:num_ex_2}
u^0(x_1, x_2)=100 (x_1+1) (x_2+1) (x_1 - 1) (x_2 - 1) (0.0625 - x_1^2 - x_2^2)
\end{equation}
on \(\Omega = [-1, 1]^2\) discretized on a $100\times 100$ grid. The initial value is shown in \Cref{fig:local_minima_mv} on the left.
\subsubsection{Inverse power method for higher-order eigenfunctions}\label{subsect:inv_it_higher_order}
The computation of higher-order eigenfunctions with  an adaption of the inverse power method for the \(p\)-Laplacian eigenproblem was analyzed in \cite{bobkov2025inverseiterationmethodhigher}.
In their work, the authors balanced the positive and negative parts of the iterate \(\zeta^k\) such that the Rayleigh quotient of the positive and negative part of \(u^{k+\nicefrac{1}{2}}\) are equal.
 The motivation behind their approach is that the second eigenfunction of the \(p\)-Laplace operator has this balancing property \cite{Bobkov2014}. 
Their proposed iteration scheme decreases the RQ, and they proved convergence of their scheme to an eigenfunction of the \(p\)-Laplacian. In their numerical experiments, they observed a plateauing effect of the RQ of the iterates in the sense that the iteration remains almost unchanged on an approximation of a higher-order eigenfunction.
After a few iterations this almost steady state of the iteration scheme becomes unstable and the same situation repeats close to another eigenfunction, until it converges towards the second eigenfunction of the \(p\)-Laplacian.

\begin{figure}[!htbp]
    \centering
    \includegraphics[width=\textwidth]{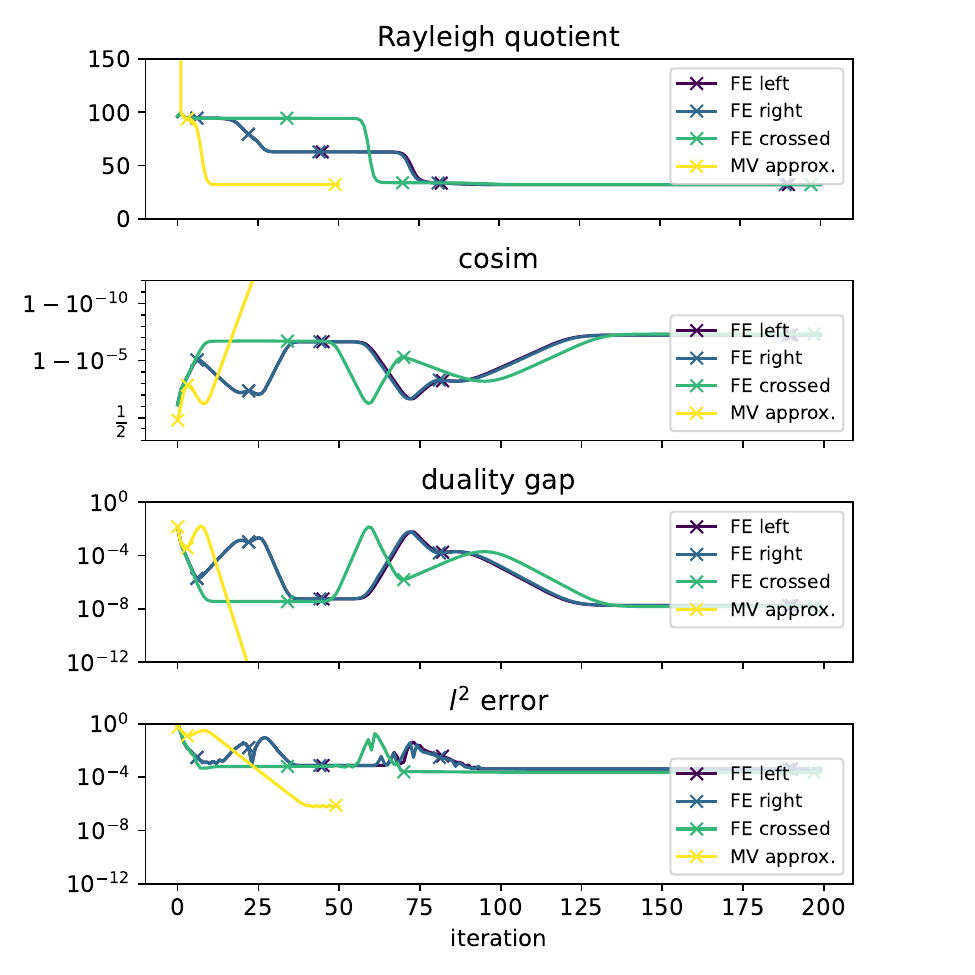}
    \caption{Evaluation of the power method for higher-order eigenfunctions.
    We evaluate the scheme with the same three  metrics as in \Cref{subsect:num_val_ipm}.
    The four plots depict the Rayleigh quotient value (top), the cosine similarity (upper center), duality gap (lower center), and \(l^2\)-error (bottom) of the iterates at every iteration.
    In the top plot, we can see that our proposed mean value approximation approach converges faster and more accurate than the finite element approaches \cite{bobkov2025inverseiterationmethodhigher} to the second smallest eigenfunction without plateauing.
    Furthermore, the `crossed' mesh, which yields a better domain tesselation, has one plateau less than the `right' and `left' mesh.
    The  values marked with `$x$' are the iterates which are local maxima of the cosine similarity and local minima of the duality gap at the same time. 
    These iterates are plotted in \Cref{fig:local_minima_fl}, \ref{fig:local_minima_fr}, and \ref{fig:local_minima_fc}.}
    \label{fig:higher_order_conv}
\end{figure}
Their implementation used the FEniCS \cite{LoggEtal2012Fenics, LoggWells2010dolfin} finite element (FE) nonlinear solver based on Newton's method for the \(p\)-Poisson problem and they used the SciPy \cite{2020SciPy-NMeth} riddler root-finding algorithm to balance the Rayleigh quotients of the positive and negative parts of \(u^{k+\nicefrac{1}{2}}\).
The authors explained the plateauing behavior of the iteration as numerical errors and assumed that with a more advanced solver of the Poisson problem the iterations would stay at the first plateau.
\begin{figure}[!htbp]
    \centering
    \includegraphics[width=\textwidth]{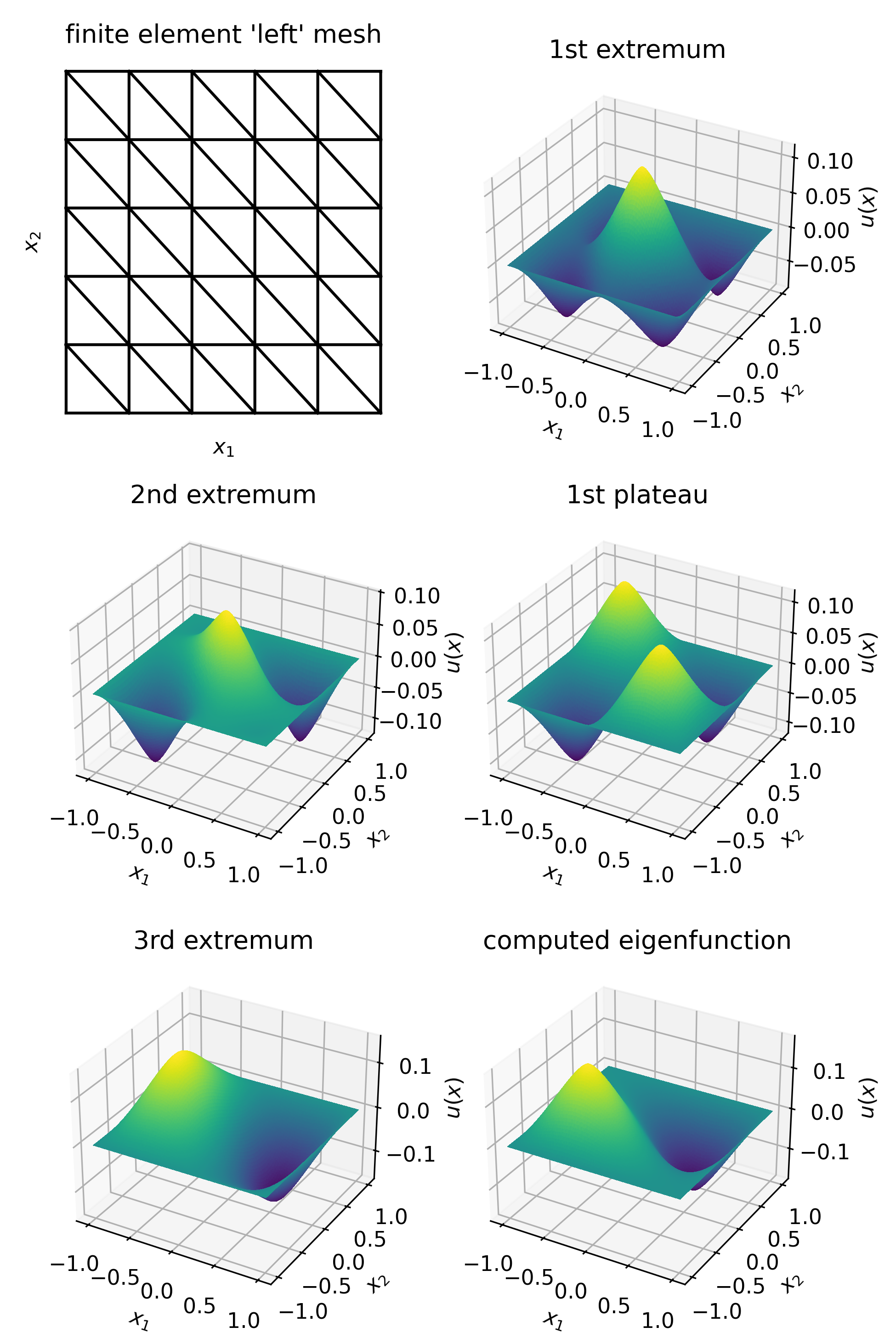}
    \caption{Top left: Mesh discretization of the finite element approximation.
    The remaining plots show the iterates marked in \Cref{fig:higher_order_conv}, which are the local extrema of the duality gap and cosine similarity of the iteration scheme.
    The first extremum coincides with the extremum of the mean value approximation.
    The second extremum only occurs in the `left' and `right' discretization. Also the first plateau appears only for  the `left' and `right' discretization. The third extremum is the final computed eigenfunction, which gets rotated by the algorithm due to the geometry of the domain and the finite elements by 90 degrees.}
    \label{fig:local_minima_fl}
\end{figure}
\begin{figure}[!htbp]
    \centering
    \includegraphics[width=\textwidth]{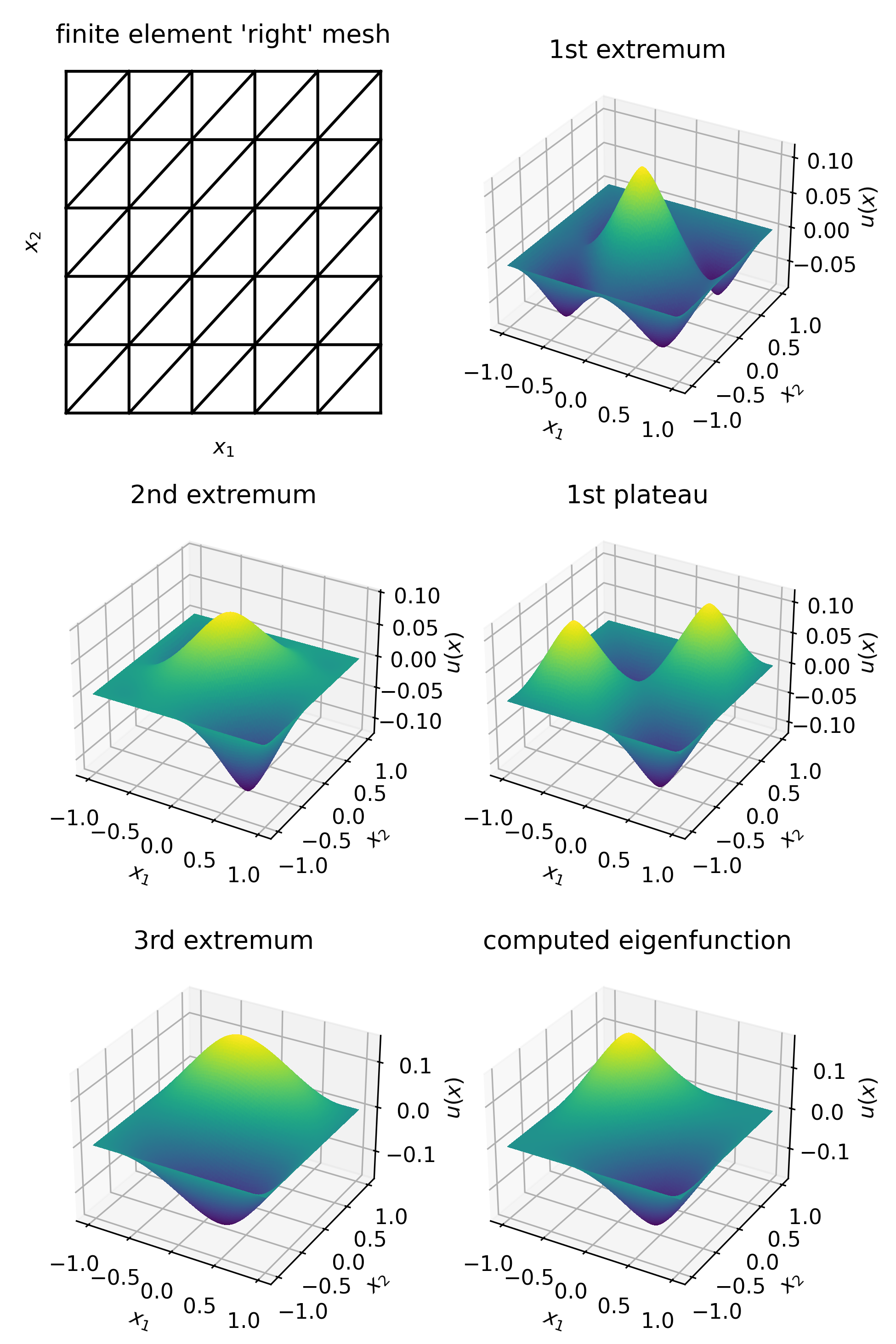}
    \caption{Top left: Mesh discretization of the finite element approximation.
    The remaining plots show the iterates marked in \Cref{fig:higher_order_conv}, which are the local extrema of the duality gap and cosine similarity of the iteration scheme.
    We observe that the extrema and plateau iterates are the same as in \Cref{fig:local_minima_fl}, but rotated by $90$ degrees.}
    \label{fig:local_minima_fr}
\end{figure}
\begin{figure}[!htbp]
    \centering
    \includegraphics[width=\textwidth]{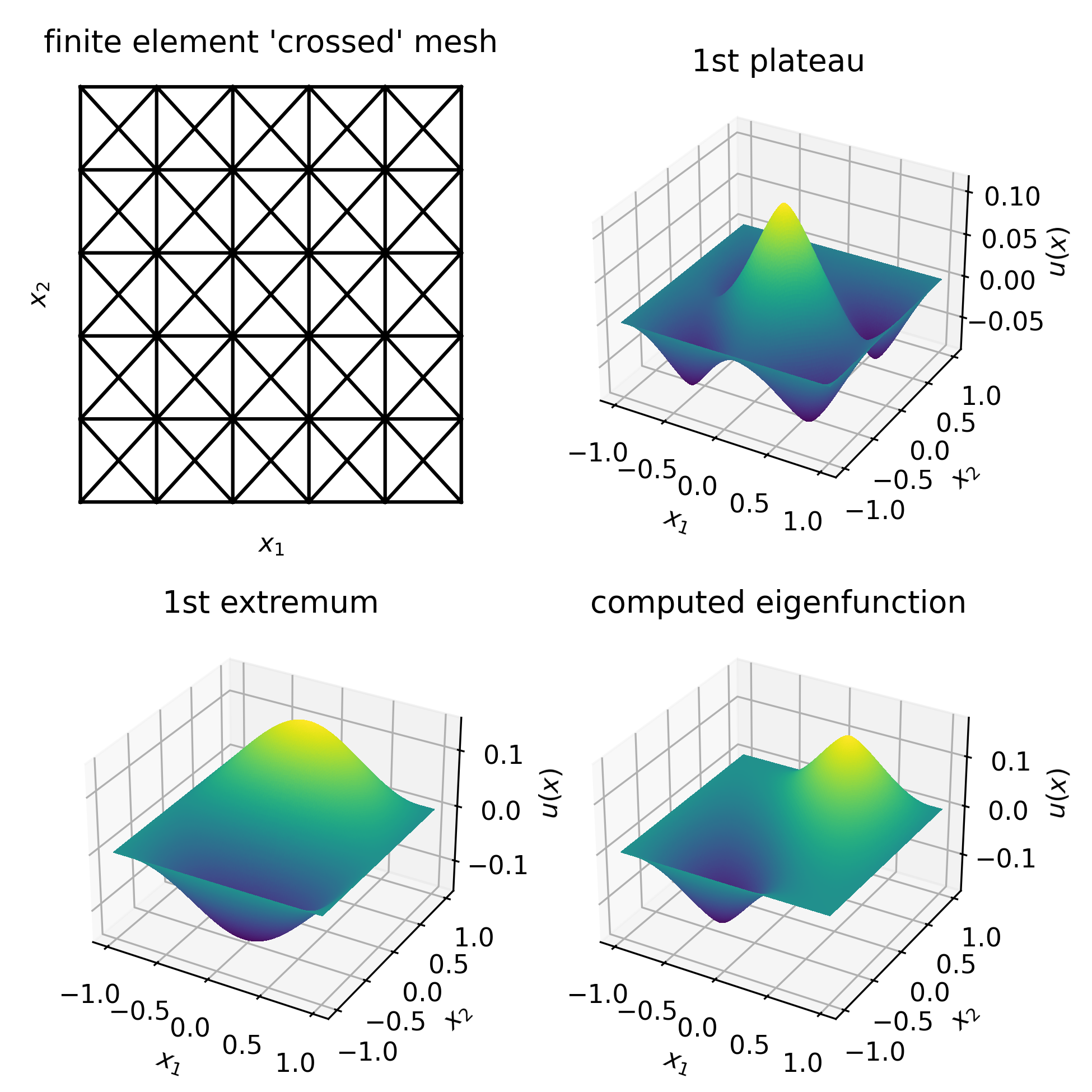}
    \caption{Top left: Mesh discretization of the finite element approximation.
    The remaining plots show the iterates marked in \Cref{fig:higher_order_conv}, which are the local extrema of the duality gap and cosine similarity of the iteration scheme.
    The first plateau is the same as the first extremum from the other schemes.
    This approach plateaus at this local extrema due to the geometry of the finite elements, which match the shape of the iterate.
    As with the other finite elements approaches, the last extremum is the final computed eigenfunction, but rotated by $90$ degrees in this case.}
    \label{fig:local_minima_fc}
\end{figure}

We implemented the proposed method in \cite{bobkov2025inverseiterationmethodhigher} with different mesh discretizations and could reproduce the plateauing behavior with the FE approach, see \Cref{fig:higher_order_conv} for the plateauing of the Rayleigh quotient value. We also implemented their algorithm with the del Teso and Lindgren finite difference solver (MV) of the \(p\)-Poisson problem based on the mean value approximation of the \(p\)-Laplacian  with radius \(r^{\frac{3}{2} + 0.1} = h\) (to conform with the theoretical consistency \eqref{eq:consistency}, see \cite{delTeso2021}) and the SciPy riddler root finding algorithm to balance the Rayleigh quotients of the positive and negative parts of \(u^{k+\nicefrac{1}{2}}\). We computed $200$ iterations for the FE approach and $50$ iterations for the finite difference method.

\Cref{fig:higher_order_conv} shows the convergence plot of our metrics for the different implemen\-tations. 
In \Cref{fig:local_minima_fl}, \ref{fig:local_minima_fr}, and \ref{fig:local_minima_fc} the almost steady states are shown.
We observe that, by changing the FE mesh, the number of iterations for which the scheme reaches plateaus changes.
For the coarser `left' and `right' mesh, the first plateau of the scheme in \Cref{fig:local_minima_fl} and \Cref{fig:local_minima_fr} is an approximation of a higher-order eigenfunction.
On the other hand, the first plateau of the finer `crossed' mesh \Cref{fig:local_minima_fc} is only a local extremum in terms of the cosine similarity and the duality gap whereas the `left' and `right', as well as the MV approach, do not stall at this plateau.
Furthermore, we observe that the solutions at this plateau seem to be a superposition of different eigenfunctions since they are only local extrema of the metrics for the MV solver.

The observation that the extrema and plateau iterates in \Cref{fig:local_minima_fr} are the same as in \Cref{fig:local_minima_fl} up to a rotation by $90$ degrees is due to the radial symmetry of the initial guess in combination with the $90$ degrees rotation of the `right' finite elements in comparison to the `left' mesh.
We deduce from this observation that the geometry of the finite elements influences the scheme and thus its local extrema.
With the MV implementation we do not observe any plateauing behavior of the RQ and we even need fewer iterations for the same amount of nodes to converge towards the same eigenfunction of the \(p\)-Laplacian, see \Cref{fig:local_minima_mv} and \ref{fig:higher_order_conv}. We expect this eigenfunction to be the second eigenfunction of the \(p\)-Laplacian since it changes its sign only once and satisfies the RQ balancing property for the positive and negative parts of the eigenfunction.
\begin{figure}[!htbp]
    \centering
    \includegraphics[width=\textwidth]{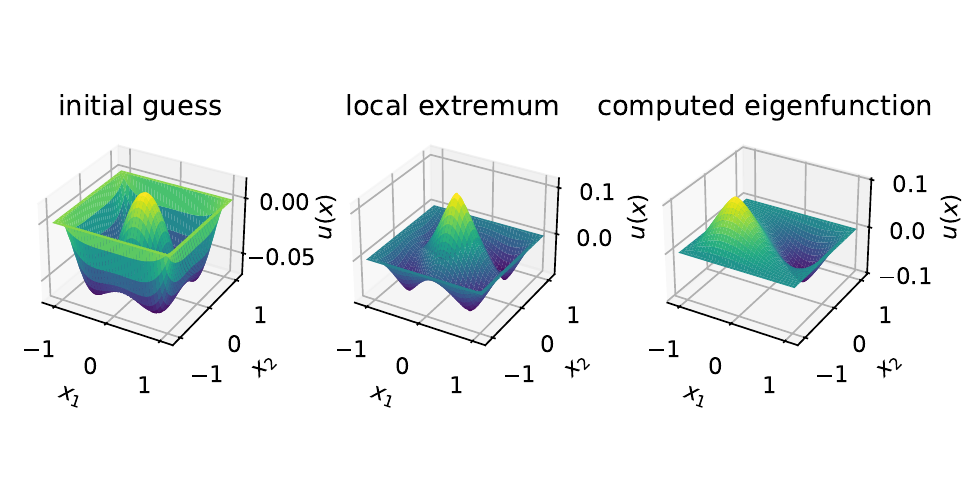}
    \caption{Left: Initial guess. Center: Iterate which fulfills the local minimum of the duality gap of the iteration from \Cref{fig:higher_order_conv} with the MV solver. Right: The final computed eigenfunction. We observe that the iteration is not plateauing, but always strictly decreasing the Rayleigh quotient.}
    \label{fig:local_minima_mv}
\end{figure}

We conclude this subsection by stating that the MV solver outperforms the FE approach with a faster convergence, a more accurate approximation, and without unwanted plateauing effects. We could also observe that the plateauing effect is not necessarily happening due to the proximity to a nonlinear eigenfunction, but could also happen due to a superposition of various eigen\-functions.
Furthermore, the plateau depends on the initial guess as observed in \cite{bobkov2025inverseiterationmethodhigher} and on the used FE mesh discretization.
\subsubsection{Higher-order eigenfunctions via geometric characterization}
The inverse power method and its adaptation for the computation of higher-order eigenfunction  in \cite{bobkov2025inverseiterationmethodhigher} provide theoretical guarantees to optimize the   dual or primal Rayleigh quotient   monotonically and to converge towards a nonlinear eigenfunction.  In the numerical experiments described above, we observe that other  metrics, like the cosine similarity or the duality gap, are not being optimized monotonically.
Furthermore, we see that the schemes based on the optimization of the RQ as a target function converge for a broad class of initial guesses to the same eigenfunction.
This eigenfunction is thus not necessarily the `closest' in terms of the cosine similarity or the duality gap to the initial guess. For absolutely \(p\)-homogeneous \(J\), given  \(\zeta\in \partial J(u)\), the  duality gap  reads as \(g(u,\zeta) = (1-\operatorname{cosim}(u,\zeta))R^{-\nicefrac{1}{p}}(u)\)  since the inequality in \Cref{lemma:RQlemma} is an equality (see comment preceding Proposition \Cref{prop:roots_dgap}). Thus,  the duality gap is directly correlated to the cosine similarity measure.

In the following, we disregard the idea of optimizing the primal or dual RQ directly, but instead we aim to optimize the cosine similarity of \(u\) and \(\zeta\in\partial J(u)\)  in order to compute higher-order nonlinear eigenfunctions.

For absolutely \(p\)-homogeneous functionals, the geometric characterization of their respective eigenvectors based on \Cref{prop:geom_charact} (see also \eqref{maximality} and \Cref{the homogeneous case}) reduces to the angular condition
\[\operatorname{cosim}(u,\zeta) = \frac{\langle \zeta, u\rangle}{|\zeta|_{H^*}|u|_H} = 1\quad \text{and }\zeta\in\partial J(u)\,.\]
Recall by \eqref{eq:euler_ident} that \(0\leq pJ(u) = \langle \zeta, u\rangle\leq |\zeta|_{H^*}|u|_H\)
for \(\zeta\in\partial J(u)\). Under the extra assumption that the functional \(J\) is differentiable, we formulate the following minimization problem:
\begin{align}\label{target function}
  \min_{u\in X}\,F(u) \quad \text{for} \quad F(u) := 1-\operatorname{cosim}(u,\partial J(u)) = 1-\frac{pJ(u)}{|u|_H|\partial J(u)|_{H^*}}. 
\end{align}
Now, due to  \Cref{prop:geom_charact}, the solution to this problem is indeed a nonlinear eigenfunction.
A numerical approach to minimize this nonconvex target function is a gradient flow discretization, for which we need to assume that the norms \(|\cdot|_H\) and \(|\cdot|_{H^*}\) are differentiable on the trajectory of the flow and \(J\) needs to be twice differentiable on the trajectory. The gradient flow is given by
\begin{equation*}
    \partial H(u'(t)) = -\nabla F(u(t))\,.
\end{equation*}
We choose a semi-implicit discretization of this flow, where we obtain a left-hand side \(\partial H(\frac{u^{k+1}-u^k}{\tau})\) and a right-hand side
\begin{equation*}\label{eq:gd_gradient}
   \frac{p\partial J(u^{k+1}) - \operatorname{cosim}(u^k,\partial J(u^k))\left[|u^k|_H^{'} |\partial J(u^k)|_{H^*} + \partial^2 J(u^k)|\partial J(u^k)|_{H^*}^{'}|u^k|_H\right]}{|u^k |_H|\partial J(u^k)|_{H^*}}
\,.
\end{equation*}
We normalize every iterate of the scheme.
We expect the scheme to increase the cosine similarity and thus to converge towards an eigenvector  or a local extremum.
The proposed scheme should converge towards the `closest' extremum in terms of the cosine similarity since it should increase the cosine similarity monotonously, in contrast to the previously discussed inverse power methods in \Cref{subsect:num_val_ipm} and \Cref{subsect:inv_it_higher_order}.
Clearly, the resulting approximation strongly depends on the chosen initial guess, so depending on it we can reach different local extrema of interest.
We analyze the numerical behavior of this scheme by applying it to the \(p\)-Laplacian eigenproblem.

\subsubsection*{Numerical implementation}

For numerical experiments on the \(p\)-Laplacian we have to approximate the Jacobian of the \(p\)-Laplacian.
As previously, we use the finite difference formulation by del Teso and Lindgren \cite{delTeso2021}.

We are using the same $100\times 100$ grid discretization with spatial step width \(h=0.02\) as in \Cref{subsect:inv_it_higher_order}. We choose a wider radius for the mean value approximation \(r=h^{0.5}\) for the \(p\)-Laplacian with \(p=3\) because our discretization needs to approximate the Jacobian of the \(p\)-Laplacian as well.
The step size is computed with a line search approach and we use Newton's method to resolve the semi-implicit update step.
The results of this method on the previous examples \eqref{eq:num_ex_1} and \eqref{eq:num_ex_2} are given in \Cref{fig:gradient_descend_results}. In \Cref{fig:gradient_descend_conv} we can see that the scheme increases the cosine similarity monotonously and decreases the duality gap and the error monotonously as well.
This approach converges very fast in less than \(10\) iterations for our examples. For the first example, we observe that the iteration scheme converges towards an eigenfunction and keeps the characteristics of the initial guess. This eigenfunction seems to be a higher eigenfunction since it changes its sign twice.
In the second example, we do not converge towards an eigenfunction, but a local extremum of the cosine similarity. This local extremum is the same superposition of eigenfunctions as in the previous approach in \Cref{subsect:inv_it_higher_order}. As a consequence, convergence towards an eigenfunction is not  guaranteed with this approach. Nonetheless, in \eqref{eq:num_ex_1} we are able to compute an eigenfunction, which is neither the first nor the second.

\begin{figure}[!htbp]
    \centering
    \includegraphics[width=\textwidth]{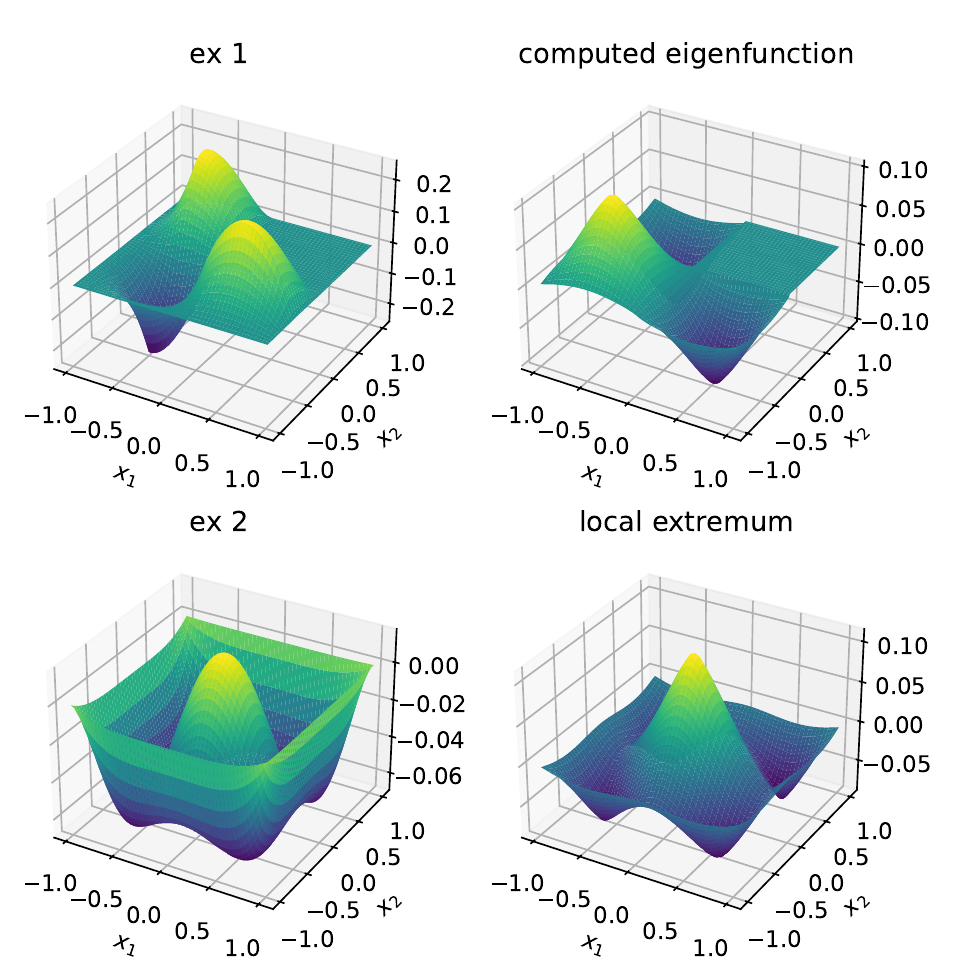}
    \caption{
       Top row: Initial guess from \eqref{eq:num_ex_1} and computed eigenfunction with the method based on the geometric characterization of eigenvectors. 
    Bottom row: Initial guess from \eqref{eq:num_ex_2} and computed local extremum of the target function, which is the same as in \Cref{subsect:inv_it_higher_order}.
    }
    \label{fig:gradient_descend_results}
\end{figure}
\begin{figure}[!htbp]
    \centering
    \includegraphics[width=\textwidth]{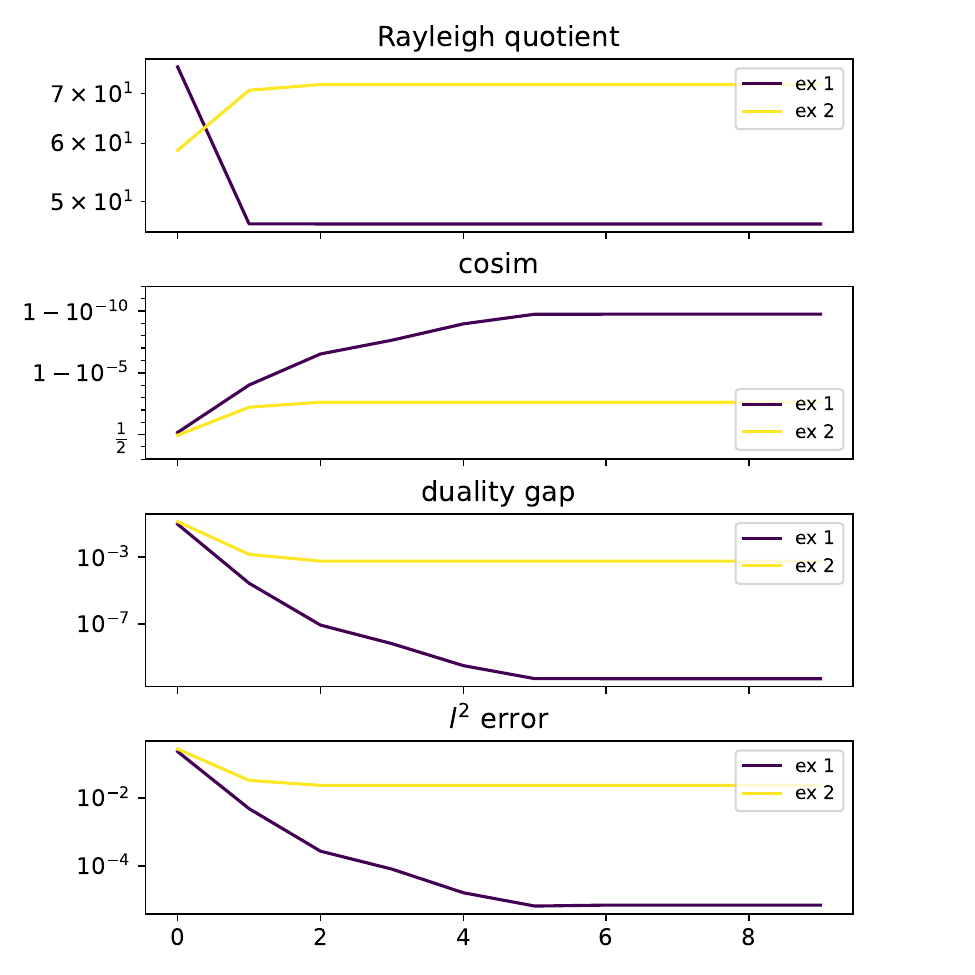}
    \caption{
       Evaluation of the method based on the geometric characterization. We evaluate the scheme with the same three  metrics as in \Cref{subsect:num_val_ipm}.
    The top plot depicts the Rayleigh quotient value of the iteration scheme which, as we see, behaves not monotonously.
    The other metrics improve monotonously. The  second example does not converge to an eigenfunction, while the first one does.
    }
    \label{fig:gradient_descend_conv}
\end{figure}

\section{Conclusion}\label{sect:Conclusion}
We conclude this work by summarizing our key findings.

In the first part of this work, we analyzed the nonlinear eigenproblem in an abstract setting for proper, closed, convex functionals in reflexive Banach spaces.
Our first main result was the equivalence of the nonlinear eigenproblem with a dual formulation in the sense that, for \(\zeta\in\partial J(u)\), \(u\) is an  eigenvector if and only if \(\zeta\) is a dual eigenvector, see \Cref{thm:duality}.
The dual problem is interpreted as the nonlinear eigenproblem of the inverse operator \(\partial J^*\) of the primal operator \(\partial J\).
This interpretation in combination with the eigenvector duality allowed us to interpret the inverse power method as a dual power method and enabled us to prove convergence towards nonlinear eigenvectors.
Further theoretical results concerned a duality gap whose roots are given by a primal dual nonlinear eigenvector pair, as well as a  novel geometric characterization of nonlinear eigenvectors. Both concepts were used as metrics for the numerical computation to observe convergence of the presented schemes towards nonlinear eigenvectors. We also used the geometric characterization to build a target function (see \eqref{target function}) in order to optimize for the computation of higher-order eigenfunctions.

Our findings in the numerical experiments showed that the inverse power method converges fast and with high accuracy towards the ground-state eigenfunction for a broad array of initial guesses.
The inverse power method for the computation of higher-order eigenfunctions computes the second eigenfunction without plateauing and with fewer iterations when using the finite difference mean value approximation solver of the \(p\)-Laplacian.
Our proposed numerical approach based on the geometric characterization  converges very fast and preserves the shape of the initial guess. However, this approach does not guarantee convergence   towards a nonlinear eigenfunction, as we could show numerically. Future work will address adaptations of the approach based on the geometric characterization, to seek for convergence guarantees towards nonlinear eigenfunctions. Further research can be done on the initial guess for this method. We believe that eigenfunctions of the \(2\)-Laplacian could be well suited as initial guesses for the computation of corresponding higher order \(p\)-Laplace eigenfunctions.
\paragraph{Acknowledgments} We thank Leon Bungert for   fruitful discussions   on the content of this paper.
\bibliographystyle{plainurl} 
\bibliography{notes_bib.bib}
\end{document}